\documentclass[a4paper,reqno]{article}

\usepackage{amsmath,amssymb,amsthm,amsfonts}
\usepackage{color}
\usepackage{hyperref}
\usepackage{bbm} 
\usepackage{mathrsfs} 
\usepackage{mathtools} 
\usepackage{enumerate}
\usepackage{textcmds} 

\newcommand{\set}[1]{{\left \{ #1 \right \}}}
\newcommand{\abs}[1]{{\left | #1 \right |}}
\newcommand{\bra}[1]{{\left( #1 \right) }} 
\newcommand{\sbra}[1]{\left[ #1 \right] }
\newcommand{\sd}{\, \cdot \,}
\newcommand{\one}{\mathbbm{1}}
\newcommand{\R}{\mathbb{R}}
\newcommand{\N}{\mathbb{N}}
\newcommand{\dx}{\:\mathrm{d}}
\newcommand{\dd}{\mathrm{d}}
\newcommand{\norm}[1]{\left\lVert #1 \right\rVert}
\newcommand{\Nfinite}{\mathbf{N}_{<\infty}}
\newcommand{\Nlf}{\mathbf{N}_{\mathrm{lf}}}
\renewcommand{\P}{\mathbb{P}}
\newcommand{\E}{\mathbb{E}}

\numberwithin{equation}{section}

\newtheorem{theorem}{Theorem}[section]
\newtheorem{lemma}[theorem]{Lemma}
\newtheorem{corollary}[theorem]{Corollary}
\newtheorem{proposition}[theorem]{Proposition}

\theoremstyle{remark}
\newtheorem{remark}[theorem]{Remark}

\theoremstyle{definition}
\newtheorem{definition}[theorem]{Definition}

\begin{document}

\title{Orthogonal Intertwiners for Infinite Particle Systems \\in the Continuum}
\date{\today}
\author{Stefan Wagner\thanks{Mathematisches Institut, Ludwig-Maximilians-Universität, Theresienstr. 39, 80333 München, Germany.} \thanks{Munich Center for Quantum Science and Technology (MCQST), Schellingstr. 4, 80799 München.}}

\maketitle

\begin{abstract}

    This article focuses on a system of sticky Brownian motions, also known as Howitt-Warren martingale problem, and correlated Brownian motions and shows that infinite-dimensional orthogonal polynomials intertwine the dynamics of infinitely many particles and their $n$-particle evolution. The proof is based on two assumptions about the model: information about the reversible measures for the $n$-particle dynamics and consistency.  Additionally, explicit formulas for the polynomials are used, including a new explicit formula for infinite-dimensional Meixner polynomials, the orthogonal polynomials with respect to the Pascal process. As an application of the intertwining relations, new reversible measures for the infinite-particle dynamics are obtained.
    
    
\end{abstract}
\tableofcontents

\section{Introduction}
Intertwining relations are a powerful tool for establishing a connection between two Markov processes. A linear operator $\Lambda$ \emph{intertwines} Markov processes $(X_t)_{t \geq 0}$ and $(Y_t)_{t \geq 0}$ with Markov semigroups $(P_t)_{t \geq 0}$ and $(Q_t)_{t \geq 0}$ if
    $\Lambda P_t = Q_t \Lambda$
holds for all $t \geq 0$.  
The concept of \emph{intertwining relations} in the context of Markov processes goes back to Dynkin \cite{DynkinIntertwining}. He employed them for the construction of new Markov semigroups based on existing ones. 
Building upon Dynkin's work, Rogers and Pitman expanded these ideas in \cite{intertwiningRogers}, which led to the characterization of Markov functions—--measurable maps that preserve the Markov property. Furthermore, intertwining relations play a crucial role in the analysis of the convergence to equilibrium in the top-to-random shuffle, see \cite{AldousDiaconis}. In \cite{Miclo2018}, intertwining relations are discussed in the context of the algebraic concept of similarity transformation. Moreover, they are used in \cite{ThomaCone} for the construction of Markov processes on infinite-dimensional spaces and in \cite{DiaconisFill} in relation to strong stationary times. Additionally, \cite{IntertwiningFractional} explores their connection with fractional operators, while \cite{IntertwiningReflected}, \cite{IntertwiningDiffusions}, \cite{AssiotisOConnelWarren} delve into their application in the context of diffusions. Intertwining relations for Ehrenfest, Yule and Ornstein-Uhlenbeck processes are studied in \cite{YuleIntertwining}. For additional examples of intertwining relations, we refer to \cite{IntertwiningMarkov} and \cite{JansenOnTheNotationOfDuality}.
Lastly, in \cite{interweaving}, the notion of interweaving relations was introduced.

Intertwining relations are closely related (see, e.g., \cite[Section~5.2]{redig_factorized_2018}, \cite[Lemma~2.1]{Groenevelt2019} and \cite{IntertwiningBetaGamma}) to the concept of stochastic duality that establishes a connection between two Markov processes using a common observable, known as the duality function (see e.g. \cite{JansenOnTheNotationOfDuality}). 
Stochastic duality has found widespread application in diverse fields, including interacting particle systems (see e.g. \cite{liggett_interacting_2005}), population genetics models (see e.g. \cite{EthierKurtz}, \cite{etheridige06},  \cite{DualitySpatiallyFlemingViot}), and (stochastic) partial differential equations (see e.g. \cite{Mueller15}). In \cite{kipnis1982heat}, duality serves as a tool for understanding the non-equilibrium steady state in the KMP model of heat conduction. Self-duality, specifically with orthogonal polynomials, is particularly valuable in studying fluctuation fields, the Boltzmann Gibbs principle, and cumulants in non-equilibrium steady states. This approach is elaborated upon in \cite{DeMasiPresutti}, \cite{ayala2018quantitative}, \cite{ayala2021higher} and \cite{floreani2020orthogonal}.

In the realm of interacting particle systems, dualities and self-dualities are typically explored in the context of lattices.
Self-duality with respect to orthogonal polynomials has been demonstrated in three prominent classical discrete interacting particle systems: the exclusion process,  inclusion process, and independent random walks, as explored in \cite{giardina2007duality}, \cite{DualityHiddenSymmetries}, \cite{franceschini2019stochastic}, \cite{Groenevelt2019}, \cite{carinci2019orthogonal}, and \cite{floreani2020orthogonal}. Self-duality means that the time evolution of an orthogonal polynomial of degree $n$ can be expressed by the time evolution of $n$ dual particles.
This is crucial because it allows us to characterize properties of a system of infinitely many particles by analyzing only a finite system. One shared property of these systems is consistency, which means that the system's time evolution commutes with the action of randomly removing a particle from the system. In \cite{carinci2021consistent}, the relationship between self-duality and consistency for reversible systems was investigated.
    
In \cite{intertwiningConsistent}, a advance was made in the study of consistent particle systems. The authors developed a general method that shows how the orthogonal dualities of these systems rely solely on the conditions of consistency and reversibility. This method does not require an explicit formula of the polynomials and works not only for particle systems on discrete spaces but also on Borel spaces like the real line $\R$. This approach produces then self-intertwiners in terms of infinite-dimensional orthogonal polynomials. 

However, the authors of \cite{intertwiningConsistent} only consider particle systems with a finite number of particles. The arguments presented in \cite{intertwiningConsistent} can apply to systems with infinitely many particles, subject to certain limitations. Specifically, to extend the analysis, it is necessary to have a reversible measure for the dynamics for such systems. Additionally, the concept of consistency alone is insufficient because removing a single particle does not link infinite and finite dynamics. However, it is worth noting that the intertwining relation \cite[Equation~(3.5)]{intertwiningConsistent} in terms of generalized falling factorial polynomials, together with reversibility, implies the intertwining relation \cite[Equation~(3.12)]{intertwiningConsistent} in terms of orthogonal polynomials.

In this article, we present a novel approach for infinite particle systems with particles on a non-discrete space that overcomes these obstacles. Our alternative setup has two key features: firstly, only reversible measures for the unlabeled $n$-particle dynamics are required, which are more manageable to obtain than reversible measures for infinite dynamics. Secondly, we define consistency in a manner that is compatible with an infinite number of particles. Given these requirements, we find that infinite-dimensional polynomials of degree $n$ intertwine the dynamics of an infinite number of particles with the dynamics of $n$ particles where $n < \infty$. Our proof relies on explicit formulas for the infinite-dimensional orthogonal polynomials.

Intertwiners for infinitely many particles have practical applications.
By using the intertwining relation that we take as the definition for consistency we obtain new invariant measures for the dynamics of infinitely many correlated or sticky Brownian motions. 
Exploiting the orthogonality of the intertwining relations in terms of infinite-dimensional polynomials, allows us to enhance this result: we prove that these measures are, in fact, reversible.

We illustrate our procedure by presenting examples of strongly consistent systems, called compatibility by Le Jan and Raimond \cite[Definition~1.1]{LeJanRaimond}. Such systems possess the property that the time-evolution of the process commutes with the removal of any deterministic particle. This characteristic, related to Kolmogorov's extension theorem, enables the construction of an infinite dynamics in a straightforward manner. 
Furthermore, there is a one-to-one correspondence between strongly consistent families and stochastic flows (see \cite{LeJanRaimond}). These flows can be interpreted as independent particles in randomly chosen environments. Brownian flows, such as the Arratia flow \cite{arratia1979coalescing,COALESCINGBROWNIANFLOWS}, the Harris flow \cite{HARRIS1984187}, or the Howitt-Warren flow \cite{howitt-warren2009}, are prominent examples of strongly consistent models. 

We focus on two examples. The first model is a system of correlated Brownian motions. We obtain intertwining relations in terms of infinite-dimensional orthogonal polynomials with respect to the distribution of the Poisson process. Their connection to the multiple Wiener-Itô integral was extensively studied in e.g. \cite{Ogura} or \cite{Surgailis}. The second example refers to a system of sticky Brownian motions, which were introduced by Howitt and Warren via a martingale problem (see \cite{howitt-warren2009}). They are observed as a scaling limit of random walks in random environments, a generalized exclusion process, or through a condensation rescaling of symmetric inclusion process (see e.g. \cite{MultidimensionalSticky}, \cite{ExactFormulas}, and \cite{CondensationSIP}). Sticky interactions are used to model colloids in materials science (see \cite{Stickysphereclusters}). In addition, these interactions have recently received attention in \cite{KPZequation}. We focus on a special case of the Howitt-Warren martingale problem called uniform sticky Brownian motions and studied in \cite{BarraquandRychnovskyLargeDeviations} and \cite{brockingtonbethe}. The multiparticle interactions in this model are entirely determined by two-particle interactions. The dynamics of this model consist of independent Brownian motions if they are apart. If they meet, they exhibit a slowly reflecting interaction behavior, as per Feller's boundary classification \cite{Feller52}. Specifically, the time during which they are equal has a positive Lebesgue measure with positive probability, but does not contain any interval. For this model, we investigate intertwining relations in terms of infinite-dimensional Meixner polynomials that are orthogonal with respect to the distribution of the Pascal process. For this latter family, we obtain a new explicit formula.
Both families of infinite-dimensional orthogonal polynomials belong to the Meixner class (\cite{meixner1934orthogonale}, \cite{LYTVYNOV2003118}). 

The article is organized as follows. In the first two sections, we introduce the two models: correlated Brownian motions (Section~\ref{section: correlated brownian motions}) and sticky Brownian motions (Section~\ref{section: sticky Brownian motions}). We introduce infinite-dimensional polynomials and present our main theorems. In Section~\ref{section: strategy}, we generalize the concept of consistency for infinitely many particles and develop a general approach that works with models beyond the two we present. In Section~\ref{section: proofs}, we prove that strongly consistent systems, including our examples, satisfy the consistency condition. We then establish the intertwining relations using the explicit formula for infinite-dimensional polynomials, and in particular, prove the explicit formula for Meixner polynomials. Finally, we examine the reversible measures for $n$-particle dynamics of correlated and sticky Brownian motion.

\section{Correlated Brownian Motions}
\label{section: correlated brownian motions}

This section focuses on intertwining relations in terms of the multiple Wiener-Itô integral with respect to the Poisson process. While previous works have focused on independent particle systems (see e.g. \cite{Surgailis}, \cite{NonEquilibriumStochasticDynamics} and \cite{KunaHydrodynamicLimits}) and consistent finite particle systems (see \cite{intertwiningConsistent}), we aim to find intertwiners for infinite particle systems with interaction. 

We consider correlated Brownian motions, a simple example within the established setting of \cite{LeJanRaimond}. Notably, this example stands out for its simplicity, as the interaction of the Brownian motions does not depend on local interactions when the particles meet, unlike other correlated Brownian motion models such as the Harris flow \cite{HARRIS1984187} or those presented in \cite{ThreeExamples}.

\subsection{The Model}
We define a family of real-valued stochastic processes $(X_{k,t})_{t \geq 0}$,  $k \in \N$ to be a family of \emph{correlated Brownian motions with pairwise correlation $0 \leq a \leq 1$} starting at a (deterministic) sequence $x = (x_k)_{k \in \N}$ of real numbers if 
\begin{itemize}
    \item The family $(X_{k,t})_{k \in \N, t \geq 0}$ is a Gaussian process.
    \item For all $t \geq 0$ and $k \in \N$: $\E \sbra{X_{k,t}} = x_k$.
    \item For all $t, s \geq 0$ and $k \in \N$: $\mathrm{Cov} \sbra{X_{k, t}, X_{k, s}} = \min \set{t,s}$. 
    \item For all $t, s \geq 0$ and $k, l \in \N$ with $k \neq l$: $\mathrm{Cov} \sbra{X_{k,t}, X_{l,s}} = a\min \set{t,s}$.  
\end{itemize}
If $a=0$, the $(X_{1,t})_{t \geq 0}, (X_{2,t})_{t \geq 0}, \ldots$ are independent Brownian motions. On the other hand, if $a=1$, they are modifications of each other up to an additive constant: for any $k \in \N$ and $t \geq 0$, we have $X_{k,t} = X_{k,1} - x_1 + x_k$ almost surely.

An explicit construction can be done as follows:
Let $(B_t)_{t \geq 0}$, $(B_{1,t})_{t \geq 0}, (B_{2,t})_{t \geq 0}, \ldots$ be independent Brownian motions all starting at zero. From there, define
\begin{align}
    \label{equation: construction correlated brownian motions}
    X_{k,t} := \sqrt{1-a} B_{k,t} + \sqrt{a} B_t + x_k, \qquad k \in \N, t \geq 0. 
\end{align}
A family of correlated Brownian motions satisfies the Markov property. We define the \emph{$n$-particle semigroup} 
\begin{align}
    \label{equation: n particle semigroup correlated brownian motions}
    P_t^{[n]} f_n(x) := \E \sbra{f_n(\sqrt{1-a} B_{1,t} + \sqrt{a} B_t + x_1, \ldots,\sqrt{1-a} B_{n,t} + \sqrt{a} B_t + x_n ) }
\end{align}
for $x = (x_1, \ldots, x_n) \in \R^n$ where $f_n : \R^n \to \R$, $n\in \N$ is bounded and measurable.

\begin{remark}
    \label{remark: strongly consistent family}
    The family of $n$-particle semigroups of correlated Brownian motions is strongly consistent (called \emph{compatibility} by \cite{LeJanRaimond}). More precisely, a family of Markov semigroups $(P_t^{[n]})_{t \geq 0}$ defined on $\R^n$, $n \in \N$, is called \emph{strongly consistent} if  the equation
    \begin{align}
        \label{equation: strong consistency}
        P_t^{[n]} f_n(x_1, \ldots, x_n) = P_t^{[l]} g_l(x_{i_1}, \ldots, x_{i_l}), \qquad x_1, \ldots, x_n \in \R
    \end{align}
    holds for all $l \leq n$, $i_1, \ldots, i_l \in \set{1, \ldots, n}$ pairwise different and $g_l : \R^l \to \R$ bounded and measurable, where $f_n : \R^n \to \R$, $(x_1, \ldots, x_n) \mapsto g_l(x_{i_1}, \ldots, x_{i_l})$.

    We remark that for a general strongly consistent family $(P_t^{[n]})_{t \geq 0}$, $n \in \N$, the existence of a Markov family $\bra{\Omega, \mathcal{F}, (X_t)_{t \geq 0}, (\P_x)_{x \in \R^\infty}}$, $X_t = (X_{k,t})_{k \in \N}$ describing the evolution of infinitely many particles is implied by Kolmogorov’s theorem (see e.g. \cite[Section 1.5.3]{LeJanRaimond}). Thereby, we denote by $\R^\infty$ the set of sequences of real numbers, and equip it with its cylindrical $\sigma$-algebra. $\P_x$ is a probability measure on a probability space $(\Omega, \mathcal{F})$ for each $x \in \R^\infty$ and $X_{k,t}$ is a real-valued random variable for $k \in \N$, $t \geq 0$. $X_0 = x$ holds $\P_x$-almost surely for all $x \in \R^\infty$, the mapping $x \mapsto \P_x[X_t \in A]$ is measurable for all measurable $A \subset \R^\infty$ and $t \geq 0$, and $(X_t)_{t \geq 0}$ satisfies the Markov property with respect to its natural filtration. This family is such that each finite subconfiguration of $l$-particles evolves according to $P_t^{[l]}$. More precisely, for each family $i_1, \ldots, i_l \in \N, l \in \N$ of pairwise different indices, $t \geq 0$ and $x = (x_k)_{k \in \N} \in \R^\infty$, the distribution of $\bra{X_{i_1,t}, \ldots, X_{i_l, t}}$ under $\P_x$ is equal to $P_t^{[l]}((x_{i_1}, \ldots, x_{i_l}), \sd)$. Using standard measure-theoretic arguments, this property extends to infinitely many particles as well: for every injection $s : \N \to \N$, the distribution of $(X_{s(k),t})_{k \in \N}$ under $\P_x$ is equal to the distribution of $(X_{k,t})_{k \in \N}$ under $\P_{(x_{s(k)})_{k \in \N}}$. 

    The correspondence between strongly consistent families and stochastic flows, as outlined in \cite{LeJanRaimond}, is evident in our example of correlated Brownian motions. Indeed, let $Z$ be a random variable that follows the standard normal distribution and let $K_t$ be the random probability kernel such that $K_t(v, \sd)$ is equal to the normal distribution with expected value $\sqrt{at} Z + v$ and variance $(1-a) t$ for $v \in \R$ and $t \geq 0$. By using \eqref{equation: n particle semigroup correlated brownian motions}, we obtain
    \begin{align*}
        P_t^{[n]} f_n(x_1, \ldots, x_n) = \E \sbra{ \int \cdots \int f_n(y_1, \ldots, y_n) K_t(x_1, \dd y_1) \cdots K_t(x_1, \dd y_1) }
    \end{align*}
    for all $t \geq 0$ and $x_1, \ldots, x_n \in \R$.
\end{remark}

\subsection{Unlabeled Dynamics}
Next, we convert the notation for labeled particle systems to the modern point process notation modeling particle configurations on $\R$ using counting measures (see e.g. \cite{LastPenroseLectures}), which employs an unlabeled notation. Let $\mathbf{N}$ denote the space of \emph{counting measures}, i.e., the space of countable sums of measures that assign values in $\N_0$ to every measurable $B \subset \R$. We equip $\mathbf{N}$ with the smallest $\sigma$-algebra $\mathcal{N}$ such that $\mathbf{N} \ni \mu \mapsto \mu(B)$ is measurable for each measurable $B \subset \R$. Since $\R$ is a Borel space, every $\mu \in \mathbf{N}$ is of the form $\mu = \sum_{k=1}^n \delta_{x_k}$ with $n \in \N \cup \set{\infty}$, $x_k \in \R$, where $\delta_x$ denotes the Dirac measure given $x \in \R$, or is equal to the zero measure, which we interpret as the empty configuration, see \cite[Section 1.1]{Kallenberg2017} or \cite[Chapter 6]{LastPenroseLectures}. In particular, $\mu(\R)$ corresponds to the total number of particles. We denote by $\Nfinite := \set{\mu \in \mathbf{N} : \mu(\R) <\infty }$ the set of \emph{finite configurations} and by $\mathbf{N}_n := \set{\mu \in \mathbf{N} : \mu(\R) = n}$ the set of configurations consisting of exactly $n \in \N_0$ particles.

For our purpose, a \emph{Markov family with state space $\mathbf{N}$} is a collection $(\Omega$, $\mathcal{F}$, $(\eta_t)_{t \geq 0}$, $(\P_\mu)_{\mu \in \mathbf{N}})$, consisting of a measurable space $(\Omega, \mathcal{F})$, measurable mappings $\eta_t: (\Omega, \mathcal{F})\to (\mathbf{N}, \mathcal{N})$, $t \geq 0$ and probability measures $\P_\mu, \mu \in \mathbf{N}$ on $(\Omega, \mathcal{F})$ such that 
\begin{itemize}
    \item For each $\mu \in \mathbf{N}$, $\P_\mu[\eta_0=\mu] = 1$.
    \item For each $n \in \N \cup \set{\infty}$, $t \geq 0$ and $A \in \mathcal{N}$, the mapping 
    \begin{align}
        \label{equation: measurability iota}
        \R^n \ni x \mapsto \P_{\iota_n(x)}[\eta_t \in A]
    \end{align}
    is measurable where $\iota_n(x) := \sum_{k=1}^n \delta_{x_k}$, $x = (x_k)_{k=1}^n \in \R^n$.
    \item The Markov property is implicitly assumed to be satisfied with respect to the natural filtration $\mathcal{F}_t := \sigma(\eta_s: 0 \leq s \leq t)$. 
\end{itemize}  
We use the notation $\E_\mu$ for the expected value with respect to the probability measure $\P_\mu$.

\begin{lemma}
\label{lemma: correlated Brownian to unlabeled}
For each $0 \leq a \leq 1$, there exists a Markov family $(\Omega$, $\mathcal{F}$, $(\eta_t)_{t \geq 0}$, $(\P_\mu)_{\mu \in \mathbf{N}})$ with state space $\mathbf{N}$ such that $\eta_t$, $t \geq 0$ is proper and describes the evolution of an unlabeled system of correlated Brownian motions with pairwise correlation $a$. More precisely, for each $\mu = \sum_{k=1}^n \delta_{x_k}$, $x_k \in \R$, $n \in \N \cup \set{\infty}$ the distribution of $(\eta_t)_{t \geq 0}$ under $\P_\mu$ is equal to the distribution of $\sum_{k=1}^n \delta_{X_{k,t}}$ where each $(X_{k,t})_{t \geq 0}$ is starting at $x_k$. 
\end{lemma}
A random measure $\zeta \in \mathbf{N}$ is said to be \emph{proper} if there exist random variables $Z_k \in \R$, $K \in \N_0 \cup \set{\infty}$ such that
\begin{align}
    \label{equation: zeta proper point process}
    \zeta = \sum_{k=1}^K \delta_{Z_k}
\end{align}
is satisfied. For more details on proper point processes, we refer to \cite{LastPenroseLectures}. 

In Lemma~\ref{lemma: correlated Brownian to unlabeled}, when $n=0$, the summation is considered to be zero. This implies that the empty configuration remains empty at all times. The Markov family can be constructed directly. Only a proof of the Markov property is required, which can be found in Section~\ref{section: proof consistent} below.

\begin{remark}
    For all $n < \infty$ and $t \geq 0$, the measurability of $\mathbf{N}_n \ni \mu \mapsto \P_{\mu}[\eta_t \in A]$ is equivalent to the measurability of \eqref{equation: measurability iota}. For $n = \infty$ the requirement of \eqref{equation: measurability iota} is in general weaker compared to  measurability of the mapping
    \begin{align}
        \label{equation: measurability infinitely many particles}
        \set{\mu \in \mathbf{N} : \mu(\R) = \infty} \ni \mu \mapsto \P_{\mu}[\eta_t \in A]. 
    \end{align}  
    With this approach, we can easily convert the concept of labeled particles into that of unlabeled particles without requiring a detailed analysis of the intricate issue of measurability of \eqref{equation: measurability infinitely many particles}. In particular, in subsequent sections of the article, we always assume that the unlabeled process starts with the distribution of a proper point process.
\end{remark}

\subsection{Multiple Wiener-Itô Integrals for the Poisson Process}
For the sake of completeness, let us recapitulate the multiple Wiener-Itô integral corresponding to the Poisson process, see e.g. \cite{Ogura} or \cite{Surgailis}. The \emph{$k$-th factorial measure} (see e.g. \cite[Eq.~(4.5)]{LastPenroseLectures}) of $\mu=\sum_{i=1}^m \delta_{x_i} \in \mathbf{N}$, $x_i \in \R$, $m \in \N_0 \cup \set{\infty}$ is given by
\begin{align}
\label{equation: factorial measure equals sum}
    \mu^{(k)} = \sum_{\substack{i_1, \ldots, i_k = 1 \\ \text{pairwise different}}}^m \delta_{(x_{i_1}, \ldots, x_{i_k})},
\end{align}
where an empty sum is considered to be equal to the zero measure. The \emph{symmetrization} of a function $f_n : \R^n \to \R$ is defined by taking the average of $f_n$ over all permutations of the coordinates (see e.g. \cite[Eq.~(27)]{LastStochasticAnalysisPoissonProcesses}). In other words, if $\mathfrak{S}_n$ denotes the set of permutations of the set $\set{1, \ldots, n}$, then the symmetrization of $f_n$ is given by
\begin{align}
    \label{equation: definition symmetrization}
    \widetilde{f_n}(x_1,\ldots, x_n) := \frac{1}{n!} \sum_{s \in \mathfrak{S}_n} f_n(x_{s(1)},\ldots, x_{s(n)}), \qquad x_1, \ldots, x_n \in \R. 
\end{align}
A function $f_n$ is called symmetric if $f_n(x_1,\ldots, x_n) = f_n(x_{s(1)}, \ldots, x_{s(n)})$ for all $s \in \mathfrak{S}_n$ and $x_1, \ldots, x_n \in \R$, i.e., $f_n$ is equal to its symmetrization. 

Let $\lambda$ denote the Lebesgue measure on $\R$. The \emph{multiple Wiener-Itô integral of degree $n$} with respect to the Poisson process with intensity measure $\lambda$ (see \cite[Section 12.2]{LastPenroseLectures} or \cite{LastStochasticAnalysisPoissonProcesses}) is defined by
\begin{align}
    \label{equation: orthogonal polynomial: poisson case}
    \notag
    I_n f_n(\mu) := &\sum_{k=0}^n \binom{n}{k} (-1)^{n-k} 
    \\
    &\hspace{4em} \iint \widetilde{f_n}(x_1, \ldots, x_n) \lambda^{\otimes (n-k)}(\dd(x_{k+1}, \ldots, x_n))  \mu^{(k)}(\dd(x_1, \ldots, x_k))
\end{align}
for $\mu \in \Nlf$ and $f_n \in \mathcal{C}_n$. 
$\Nlf$ denotes the set of \emph{locally finite} $\mu \in \mathbf{N}$, i.e., $\mu(B) < \infty$ for measurable, bounded $B \subset \R$ while $\mathcal{C}_n$ denotes the space of bounded, measurable functions $f_n : \R^n \to \R$ with bounded support. 
Furthermore, when integrating with respect to $\mu^{(0)}$ or $\lambda^{\otimes 0}$, we treat the integrals as if they were absent. Specifically, we obtain $I_0 c(\mu) = c$, for $c \in \mathcal{C}_0 := \R$.

The orthogonality relation
\begin{align}
    \label{equation: orthogonality relation poisson}
    \int (I_n f_n) (I_m g_m) \dx \pi_\lambda = \one_{\set{n=m}} n! \int f_n g_m \dx \lambda^{\otimes n}
\end{align}
holds for symmetric $f_n \in \mathcal{C}_n$, $g_m \in \mathcal{C}_m$ where $n, m \in \N_0$ (see e.g. \cite{Surgailis2}, \cite{LastStochasticAnalysisPoissonProcesses}). $\pi_\lambda$ denotes the distribution of the Poisson point process with intensity measure $\lambda$. Hence, $I_n$ can be uniquely extended to a bounded linear operator mapping square-integrable, symmetric functions, denoted by $L^2_\mathrm{sym}(\lambda^{\otimes n})$, to $L^2(\pi_\lambda)$. We put $L^2_{\mathrm{sym}}(\lambda^{\otimes 0}) := \R$.

Additionally, for each $F \in L^2(\pi_\lambda)$, there exists a unique sequence $(f_n)_{n \in \N_0}$ that belongs to the \emph{Fock space} $\bigoplus_{n=0}^\infty \frac{1}{n!} L^2_{\mathrm{sym}}(\lambda^{\otimes n})$, i.e., it satisfies $\sum_{n=0}^\infty \frac{1}{n!} \norm{f_n}_{L^2(\lambda^{\otimes n})}^2 < \infty$. This sequence is such that $F = \sum_{n=0}^\infty \frac{1}{n!} I_n f_n$, referred to as the \emph{chaos decomposition}. 

Furthermore, it is worth noting that $I_n f_n$ has a close relationship with \emph{Charlier polynomials} (see e.g.  \cite{Surgailis2},\cite{Last2011}) which are self-duality functions of independent random walkers (see e.g. \cite{redig_factorized_2018}, \cite{franceschini2019stochastic}). 

The study of non-Gaussian white noise naturally involves infinite-dimensional orthogonal polynomials, as discussed in \cite{berezansky1996infinite-dimensional}. In \cite{Lytvynov2003}, the relationship between chaos decompositions using multiple stochastic integrals with power jump martingales (e.g. \cite{schoutens} or \cite{nualart-schoutens2000}) and polynomial chaos is thoroughly examined.
The use of chaos decompositions is significant in the analysis of Lévy white noise and stochastic differential equations driven by Lévy white noise, as explored in \cite{dinunno-oksendal-proske2004},\cite{lokka-proske2006},\cite{meyerbrandis2008} and \cite{CalculusVariations}. In \cite[Section 3.3]{intertwiningConsistent}, intertwiners based on infinite-dimensional orthogonal polynomials have been used in the context of finite particle systems. To define the space of infinite-dimensional polynomials, we consider polynomials with bounded coefficients and bounded support to ensure their square-integrability, and interpret the functions $u_k$ in the definition below as coefficients. The set of polynomials of degree less or equal to $n \in \N_0$ is defined by
\begin{align}
    \label{equation: polynomials}
    \mathcal{P}_n := \set{\Nlf \ni \mu \mapsto u_0 + \sum_{k=1}^n \int u_k \dx \mu^{\otimes k} : u_0 \in \R, u_k \in \mathcal{C}_k}.
\end{align} 
The \emph{infinite-dimensional orthogonal polynomial of degree $n$} is defined by the orthogonal projection of $\bra{\mu \mapsto \int f_n \dx \mu^{\otimes n}}$ onto $\mathcal{P}_{n-1}^\perp$ if $n \geq 1$ and by the constant function equal to one if $n = 0$. It coincides with $I_n f_n$ $\pi_\lambda$-almost surely for all $f_n \in \mathcal{C}_n$, $n \in \N_0$.

\subsection{Main Result}
According to \cite[Corollary 3.7]{LastPenroseLectures}, there exists a proper Poisson process. Thus, let $\zeta$ be a proper Poisson process with intensity measure given by the Lebesgue measure $\lambda$. Without loss of generality, we can define $\zeta$ on the measurable space $(\Omega, \mathcal{F})$ from Lemma~\ref{lemma: correlated Brownian to unlabeled}, equipped with a probability measure $\P$ and expected value $\E$. Note that $\P \sbra{ \zeta(\R) = \infty} = 1$ due to the fact that $\lambda(\R) = \infty$.
\begin{theorem}
    \label{Theorem: correlated Brownian motions}
    Suppose $0 \leq a \leq 1$. Then, for every $n \in \N$, the multiple Wiener-Itô integral of degree $n$ intertwines the dynamics of infinitely many correlated Brownian motions with pairwise correlation $a$ and their $n$-particle evolution. In other words, 
    \begin{align}
        \label{equation: intertwining Poisson}
        \E_\zeta \sbra{I_n f_n(\eta_t)} = I_n P_t^{[n]} f_n(\zeta)
    \end{align}
    holds almost surely for all $t \geq 0$ and for all functions $f_n \in L^2(\lambda^{\otimes n})$. 
\end{theorem}
The proof of Theorem~\ref{Theorem: correlated Brownian motions} relies on two key components. The notion of consistency, which we introduce along with the abstract framework after Theorem~\ref{Theorem: abstract theorem} in Section~\ref{section: strategy} below, and the following proposition, which follows directly from the definition of correlated Brownian motions.
\begin{proposition}
    \label{Proposition: correlated n-particle reversible} For each $n \in \N$, the measure $\lambda^{\otimes n}$ is reversible for $n$ correlated Brownian motions with pairwise correlation $a$, i.e., $\int (P_t^{[n]} f_n)  g_n \dx \lambda^{\otimes n} = \int (P_t^{[n]} g_n) f_n \dx \lambda^{\otimes n}$ for $t \geq 0$, $f_n, g_n \in \mathcal{C}_n$. 
\end{proposition}
Theorem~\ref{Theorem: correlated Brownian motions} is applicable in finding reversible measures for infinitely many particles. We say that the measure $\pi_\lambda$ is \emph{reversible} for $(\eta_t)_{t \geq 0}$ if 
\begin{align}
    \label{equation: definition reversible}
    \E \sbra{ \E_\zeta \sbra{F(\eta_t)} G(\zeta) } = \E \sbra{ F(\zeta) \E_\zeta \sbra{G(\eta_t)}}
\end{align}
holds for all measurable, bounded functions $F, G: \mathbf{N} \to \R$ and $t \geq 0$. Since $\zeta$ is proper and \eqref{equation: measurability iota} is measurable, $\P_{\zeta} \sbra{F(\eta_t)}$ is measurable. Thus, the outer expected values in \eqref{equation: definition reversible} are well-defined. 

By applying standard arguments, we obtain the following corollary. Specifically, the family of $n$-particle Markov semigroups undergoes a unitary transformation, resulting in the Markov semigroup of unlabeled infinite dynamics. This transformation is the only significant step in the proof.

\begin{corollary}
    \label{corollary: correlated Brownian motions}
    The measure $\pi_\lambda$, which denotes the distribution of the Poisson process with intensity measure $\lambda$, is reversible for $(\eta_t)_{t \geq 0}$, which is an infinite system of unlabeled correlated Brownian motions with pairwise correlation $a$. 
\begin{proof}
    Consider the unitary operator
    \begin{align*}
        \mathfrak{U} : \mathfrak{F} := \bigoplus_{n=0}^\infty \frac{1}{n!} L^2_{\mathrm{sym}}(\lambda^{\otimes n}) \to L^2(\pi_\lambda), \qquad (f_n)_{n \in \N_0} \mapsto \sum_{n=0}^\infty \frac{1}{n!} I_n f_n,
    \end{align*}
    and for each $t\geq 0$, define the operator $P_t : \mathfrak{F} \to \mathfrak{F}$, $(f_n)_{n \in \N_0} \mapsto (P_t^{[n]} f_n)_{n \in \N_0}$ where $P_t^{[0]} f_0 := f_0$, $f_0 \in \R$. 
    By Proposition~\ref{Proposition: correlated n-particle reversible}, we know that $P_t$ is a well-defined, bounded, self-adjoint operator for every $t>0$. Hence, $\mathfrak{U} P_t \mathfrak{U}^{-1}$ is also self-adjoint. Therefore, to complete the proof, it suffices to show that
    \begin{align}
        \label{equation: St equals E zeta F eta t}
        \mathfrak{U} P_t \mathfrak{U}^{-1} F(\zeta) = \E_\zeta \sbra{F(\eta_t)} 
    \end{align}
    holds almost surely for all $F \in L^2(\pi_\lambda)$. Theorem~\ref{Theorem: correlated Brownian motions} implies that \eqref{equation: St equals E zeta F eta t} holds for $F = I_n f_n$, where $f_n \in L^2(\lambda^{\otimes n})$. 

    In the proof of Theorem~\ref{Theorem: abstract theorem} below we show that the operator $L^2(\pi_\lambda) \to L^2(\Omega, \mathcal{F}, \P)$, $F \mapsto \E_\zeta \sbra{F(\eta_t)}$ obtained from the right-hand side of \eqref{equation: St equals E zeta F eta t} is a contraction, as can be seen in \eqref{equation: inequality E zeta} below. Thus, by means of an approximation argument, we conclude that \eqref{equation: St equals E zeta F eta t} holds for all $F \in L^2(\pi_\lambda)$.
\end{proof}
\end{corollary}

\section{Sticky Brownian Motions}
\label{section: sticky Brownian motions}

The focus of this section is on sticky Brownian motions. Similar to correlated Brownian motions, this model leads us to infinite-dimensional orthogonal polynomials, specifically, infinite-dimensional Meixner polynomials. These polynomials serve as intertwiners, and they are orthogonal with respect to the distribution of the Pascal process. As an application of the intertwining relations, we obtain a new result for a system of infinitely many sticky Brownian motions: the distribution of the Pascal process is reversible.

Intertwining relations involving infinite-dimensional Meixner polynomials were already studied in \cite[Section 5]{intertwiningConsistent}. The authors examined a version of the symmetric inclusion processes in the continuum. 

\subsection{The Model}

Feller considered in his boundary classification (see \cite{Feller52}) a single reflected Brownian motion that is sticky at the origin at a rate $\theta > 0$. However, it should be noted that this condition is more accurately described as a \qq{slowly} reflecting boundary, rather than \qq{true} stickiness. Specifically, the amount of time that the Brownian motion spends at zero has a strictly positive Lebesgue measure with positive probability, but contains no interval. In contrast, \qq{true} stickiness is characterized by a process getting \qq{stuck} at zero over a random interval of time, as shown e.g. \cite{kabanov2007positive} for an interest rate process.

Based on that, a pair $(X_{1,t}, X_{2,t})_{t \geq 0}$ of Brownian motions with sticky interaction can be described, namely that $(X_{1,t})_{t \geq 0}$ and $(X_{2,t})_{t \geq 0}$ are both Brownian motions and $(\abs{X_{1,t} - X_{2,t}})_{t \geq 0}$ is a single reflected Brownian motion that is sticky at zero, in the sense of Feller. They behave independently when they are apart and interact when they meet. The pair $(X_{1,t}, X_{2,t})_{t \geq 0}$ can be characterized by a martingale problem, as shown in \cite{howitt-warren2009}.

Howitt and Warren (see \cite{howitt-warren2009}) generalized this concept to a family of $n$ diffusions on $\R$, which is commonly known as the \emph{Howitt-Warren martingale problem}. This yields a family of $n$ independent Brownian motions that move separately when they are far apart and coalesce when individual processes meet. A new non-negative quantity $\theta(i:j)$ is introduced, which can be interpreted as the rate at which a group of $i+j$ particles splits into a group of $i$ particles and a group of $j$ particles. Furthermore, the concept of strong consistency (see Remark~\ref{remark: strongly consistent family}) corresponds to the condition $\theta(i+1: j) + \theta(i: j+1) = \theta(i: j)$, which has been characterized in \cite[Lemma A.4]{schertzer2014stochastic} by the existence of a finite measure $\nu$ on the interval $[0,1]$, called \emph{the characteristic measure}, which satisfies
\begin{align*}
    \theta(i:j) = \int x^{i-1} (1-x)^{j-1} \nu(\dd x), \qquad i, j \in \N. 
\end{align*}
For a full construction using the martingale problem, we refer to \cite{howitt-warren2009}. An alternative formulation is available in \cite{schertzer2014stochastic}.

In this article, we examine uniform sticky Brownian motions with stickiness $\theta > 0$ and zero drift, which were studied by \cite{BarraquandRychnovskyLargeDeviations} in the context of large deviation analysis. This model is defined by choosing the characteristic measure $\nu = \frac{\theta}{2} \lambda_{[0,1]}$, where $\lambda_{[0,1]}$ is the Lebesgue measure on $[0, 1]$, leading to 
\begin{align*}
    \theta(i:j) = \frac{\theta}{2} \frac{(i-1)! (j-1)!}{(i+j-2)!}. 
\end{align*}
For this model, the multiparticle interactions are completely determined by the two-particle interactions. The authors of \cite{brockingtonbethe} derive the Kolmogorov backwards equation and demonstrate that, for this particular interaction, it can be solved exactly using the Bethe ansatz. This result is used to obtain the reversible measure for the $n$-particle dynamics. We provide the definition of the uniform sticky Brownian motion from \cite[Definition 2.2]{schertzer2014stochastic} for the sake of completeness.
\begin{definition}
    Let $n \in \N$. We say that $(X_t)_{t \geq 0} = (X_1, \ldots, X_n) = (X_{1,t}, \ldots, X_{n,t})_{t \geq 0}$ are \emph{$n$-particle uniform sticky Brownian motions with stickiness $\theta > 0$} if the following conditions are satisfied. 
    \begin{enumerate}[\normalfont(i)]
        \item $(X_t)_{t \geq 0}$ is a continuous, square-integrable semimartingale. 
        \item The covariation of $X_k$ and $X_l$ is given by
        \begin{align*}
            [X_k, X_l]_t = \int_{0}^t \one_{\set{X_{k,s} = X_{l,s}}} \dx s, \qquad t \geq 0
        \end{align*}
        for $k, l \in \set{1, \ldots, n}$.
        \item For each $\Delta \subset \set{1, \ldots, n}$, $f_\Delta(X_t) - \theta \int_{0}^t \beta_+(g_\Delta(X_s)) \dx s$, $t \geq 0$ is a martingale with respect to the natural filtration of $(X_t)_{t \geq 0}$ where 
        \begin{align*}
            f_\Delta(x) := \max_{k \in \Delta} x_k , \qquad g_\Delta(x) := \abs{\set{k \in \Delta : x_k = f_{\Delta}(x)}}, \qquad x = (x_1, \ldots, x_n) \in \R^n,
        \end{align*}
        $\beta^+(1) := 0$ and $\beta^+(m) := 1 + \frac{1}{2} + \frac{1}{3} + \ldots \frac{1}{m-1}$, $m \geq 2$. 
    \end{enumerate}
\end{definition}
The existence of the $n$-particle uniform sticky Brownian motions starting at an arbitrary initial value $x \in \R^n$ was proven in \cite[Theorem 2.1.]{howitt-warren2009} together with the fact, that their distribution is unique. A family of sticky Brownian motions satisfies the Markov property. Additionally, it is shown that the corresponding $n$-particle semigroups form a strongly consistent family, as described in Remark~\ref{remark: strongly consistent family}.

\subsection{Pascal Process and Infinite-Dimensional Meixner Polynomials}

Let $\alpha := \theta \lambda$, where $\lambda$ denotes the Lebesgue measure on $\R$ and $\theta > 0$ is fixed. The Pascal process, also known as the negative binomial process, is a well-studied point process (see \cite{serfozo1990point}, \cite{berezansky2002pascal} or \cite{kozubowski2008distributional} for further details). We briefly recall the construction of a Pascal process. A \emph{Pascal process} with parameters $0 < p < 1$ and $\alpha$ is a random measure $\zeta \in \mathbf{N}$ satisfying the following properties.

\begin{enumerate}[\normalfont(i)]
    \item The random variables $\zeta(A_1), \ldots, \zeta(A_N)$ are independent if the sets $A_1, \ldots, A_N \subset \R$ are measurable and pairwise disjoint. 
    \item For each measurable $A \subset \R$ such that $0 < \alpha(A) < \infty$, the random variable $\zeta(A)$ follows the negative binomial distribution with parameters $p$ and $\alpha(A)$, i.e.,
    \begin{align*}
        \P[\zeta(A) = k] = \alpha(A) (\alpha(A) + 1) \cdots (\alpha(A) + k-1) \frac{p^k}{k!} (1-p)^{\alpha(A)}, \qquad k \in \N_0.
    \end{align*}
    If $\alpha(A) = 0$, then $\zeta(A) = 0$ almost surely. If $\alpha(A) = \infty$, then $\zeta(A) = \infty$ almost surely. 
\end{enumerate}
The Pascal process is a compound Poisson process. Specifically, if we consider a Poisson process $\xi$ on $\R \times \N$ with intensity measure $\alpha \otimes \sum_{k=1}^\infty \frac{p^k}{k} \delta_k$, then $\zeta(A) := \int_{A \times \N} y \xi(\dd(x,y))$ defines a Pascal process. In particular, the Pascal process can be constructed as a proper point process, meaning it has the form \eqref{equation: zeta proper point process}. For more details on Pascal processes, see also  \cite{CompletelyRandomMeasures} and \cite{intertwiningConsistent}. 

As a reminder, the symmetrization $\widetilde{f_n}$ of a function $f_n$ is defined as in \eqref{equation: definition symmetrization}. We put for each $\mu \in \Nlf$ and $f_n \in \mathcal{C}_n$ 
\begin{align}
    \label{equation: infinite-dimensional Meixner polynomials}
    \mathcal{M}^{p, \alpha}_n f_n(\mu) &:= \sum_{k=0}^n  \binom{n}{k} \bra{1-\frac{1}{p}}^{k-n}  \iint \widetilde{f_n}(x_1, \ldots, x_n) \notag \\
    &\hspace{4em}\kappa_{n,k}(x_1, \ldots, x_k, \dd(x_{k+1}, \ldots, x_n)) \mu^{(k)}(\dd (x_1, \ldots, x_k)),
\end{align}
the \emph{infinite-dimensional Meixner polynomial of degree $n$}. The kernel $\kappa_{n,n-1} : \R^{n-1} \times \mathcal{B}(\R) \to [0, \infty) \cup \set{\infty}$ is defined by
\begin{align*}
    \kappa_{n, n-1}((x_1, \ldots, x_{n-1}), \sd) := \alpha + \delta_{x_1} + \ldots + \delta_{x_{n-1}}, 
\end{align*}
while $\kappa_{n, k} : \R^k \times \mathcal{B}(\R^{n-k}) \to [0, \infty) \cup \set{\infty}$ is defined by $\kappa_{n, k} := \kappa_{k+1, k} \otimes \kappa_{k+2, k+1} \otimes \cdots \otimes \kappa_{n, n-1}$. In other words,  
\begin{align}
    \label{equation: definition kappa}
    \notag
    &\kappa_{n, k}((x_1, \ldots, x_k), \dd( x_{k+1}, \ldots, x_n)) = (\alpha + \delta_{x_1} + \ldots + \delta_{x_{n-1}})(\dd x_n) \\
    &\hspace{8em} \cdots (\alpha + \delta_{x_1} + \ldots + \delta_{x_{k+1}})(\dd x_{k+2}) (\alpha + \delta_{x_1} + \ldots + \delta_{x_k})(\dd x_{k+1}).
\end{align}
We interpret the case $n = k$ in \eqref{equation: infinite-dimensional Meixner polynomials} as though the inner integral was not present. If $k=0$, $\kappa_{n,0}$ reduces to a measure on $\R^n$, and we interpret the outer integral with respect to $\mu^{(0)}$ as though it was not present. If $n = 0$, we put $\mathcal{M}^{p, \alpha}_n f_0 := f_0$ for $f_0 \in \mathcal{C}_0 = \R$. Let $\mathcal{P}_n$ denote the set of polynomials of degree less than or equal to $n$, as defined by \eqref{equation: polynomials}.

Proposition~\ref{proposition: explicit formula infinite-dimensional Meixner} below demonstrates that the polynomials $\mathcal{M}_n^{p, \alpha} f_n$ defined in \eqref{equation: infinite-dimensional Meixner polynomials} coincide with the infinite-dimensional orthogonal polynomials with respect to the distribution of the Pascal process. These polynomials are investigated, e.g., in \cite[Section~5 and~6]{Lytvynov2003} and \cite[Section~5.2]{intertwiningConsistent}. The proof of this proposition is provided in Section~\ref{section: explicit formula Meixner} below. 
Equation~\eqref{equation: infinite-dimensional Meixner polynomials} is a novel and explicit formula for the infinite-dimensional orthogonal polynomials following a format similar to that of the Wiener-Itô integrals \eqref{equation: orthogonal polynomial: poisson case}. Notably, our formula differs from the one presented in \cite[Equation~(6.4)]{Lytvynov2003}, which only provides a recursion. 
Our explicit formula is crucial in the proof of Theorem~\ref{Theorem: sticky Brownian motions} below, where we exploit its structure as a sum of integrals involving factorial measures and the kernels $\kappa_{n,k}$. Furthermore, we use Proposition~\ref{proposition: explicit formula infinite-dimensional Meixner} in that proof and in the proof of Corollary~\ref{corollary: sticky Brownian motions} below, more precisely, we exploit the orthogonality relation~\eqref{equation: orthogonality infinite dimensional Meixner} below. 

\begin{proposition}
    \label{proposition: explicit formula infinite-dimensional Meixner}
    For each $n \in \N$ and $f_n \in \mathcal{C}_n$,
    \begin{align}
        \label{equation: explicit formula infinite-dimensional Meixner}
        \mathcal{M}^{p, \alpha}_n f_n = \text{orthogonal projection of } \bra{\mu \mapsto \mu^{\otimes n}(f_n)} \text{ onto } \mathcal{P}_{n-1}^\perp
    \end{align}
    holds $\rho_{p, \alpha}$-almost surely. 
\end{proposition}

\begin{remark}
    \label{remark: univariate meixner}
    The infinite-dimensional Meixner polynomials, given by \eqref{equation: infinite-dimensional Meixner polynomials}, can be regarded as a natural extension of the monic univariate Meixner polynomials with parameters $0 < p < 1$ and $\alpha > 0$. The \emph{monic Meixner polynomials} are expressed as follows (e.g. \cite{HypergeometricOrthogonalPolynomials}):
    \begin{align}
        \label{equation monic Meixner}
        \mathscr{M}^{p, a}_n(x) = \sum_{k=0}^n \binom{n}{k} \bra{1-\frac{1}{p}}^{k-n} (\alpha + k)^{(n-k)} (x)_k, \qquad n, x \in \N_0.
    \end{align}
    \label{remark: univariate Meixner}
    In the infinite-dimensional version of these polynomials, the \emph{rising factorial} (also known as the \emph{Pochhammer symbol}), $(a + k)^{(n-k)} = (a+n-1) \cdots (a + k + 1) (a + k)$, where $(a + k)^{(0)} = 1$, turns into an integration with respect to the kernel $\kappa_{n, k}$, while the \emph{falling factorial}, $(x)_k = x (x-1) \cdots (x-k+1)$, where $(x)_0 := 1$, occurs as an integration with respect to the factorial measure $\mu^{(k)}$. These relationships imply that if $f_n$ is wisely chosen, \eqref{equation: infinite-dimensional Meixner polynomials} reduces to a product of univariate polynomials, as shown in Lemma~\ref{lemma: explicit meixner product formula} below. 
\end{remark}

We denote by $\Sigma_n$ the set of partitions $\sigma$ of the set $\set{1, \ldots, n}$. For a function $f_n : \R^n \to \R$ denote by $(f_n)_{\sigma} : \R^{\abs{\sigma}} \to \R$ the function gained by identifying the variables belonging to the same $A \in \sigma$ in the order of occurrence. Put
\begin{align}
    \label{equation: definition lambda n}
    \lambda_n := \sum_{\sigma \in \Sigma_n} \bra{\prod_{A \in \sigma}  (\abs{A}-1)!}  \alpha_\sigma
\end{align}
where $\alpha_\sigma$ is the measure defined by $\int f_n \dx \alpha_\sigma = \int (f_n)_\sigma \dx \alpha^{\otimes \abs{\sigma}}$. Note that
\begin{align}
    \label{equation: lambda n equals kappa lambda k}
    \lambda_n = \kappa_{n,0} = \lambda_k \otimes \kappa_{n, k}
\end{align} 
for all $n \geq k \geq 0$. 
The measure $\lambda_n$ is closely related to the Pascal process in two ways. Firstly, the definition of the Pascal process implies the following proposition.
\begin{proposition}
\label{proposition: factorial moment measure Pascal process}
The $n$-th factorial moment measure of the Pascal process with parameters $p, \alpha$ is given by $(\frac{p}{1-p})^n \lambda_n$ for each $n \in \N$, i.e., $\E \sbra{\zeta^{(n)}(A)} = (\frac{p}{1-p})^n \lambda_n(A)$ for $A \subset \R^n$ measurable. 
\end{proposition}
Secondly, the orthogonality relation
\begin{align}
    \label{equation: orthogonality infinite dimensional Meixner}
    \int (\mathcal{M}^{p, \alpha}_n f_n)(\mathcal{M}^{p, \alpha}_m g_m) \dx \rho_{p, \alpha} = \one_{\set{n = m}} \frac{p^n n!}{(1-p)^{2n}} \int f_n g_m \dx \lambda_n
\end{align}
holds for symmetric $f_n \in \mathcal{C}_n$, $g_m \in \mathcal{C}_m$ (see \cite{Lytvynov2003}, \cite{intertwiningConsistent}). Hence, the linear operator $\mathcal{M}^{p, \alpha}_n$ can be extended continuously to symmetric, measurable, square-integrable functions, which are denoted by $L^2_{\mathrm{sym}}(\lambda_n)$. This extension allows us to decompose any function $F \in L^2(\rho_{p, \alpha})$ into a series that converges in $L^2(\rho_{p, \alpha})$, given by $F = \sum_{n=0}^\infty \frac{(1-p)^n}{n!} \mathcal{M}^{p, \alpha}_n f_n$. The sequence $(f_n)_{n \in \N_0}$ is contained in a space called the \emph{extended anyon Fock space}, as defined in \cite{Lytvynov2003}, \cite{LYTVYNOV2003118}, \cite{bozejko2015extended}. This space is defined as the set of sequences $(f_n)_{n \in \N_0}$ that satisfy the condition $\sum_{n=0}^\infty \frac{p^n}{n!} \norm{f_n}_{L^2_{\mathrm{sym}}(\lambda_n)}^2 < \infty$, where $L^2_{\mathrm{sym}}(\lambda_0)$ understood as $\R$. 

\subsection{Main Result}
For each $x \in \R^\infty$, consider a family of stochastic processes $(X_{k,t})_{k \in \N}$, $t \geq 0$ such that that for any pairwise different $i_1, \ldots, i_n \in \N$, the distribution of $(X_{i_1, t}, \ldots, X_{i_n, t})$, $t \geq 0$ is the same as that of the $n$-particle uniform sticky Brownian motions with stickiness $\theta > 0$ starting at $(x_{i_1}, \ldots, x_{i_n})$. That sequence exists using strong consistency and Kolmogorov’s theorem. Let $(P_t^{[n]})_{t \geq 0}$ denote the Markov semigroup of the $n$-particle uniform sticky Brownian motions with stickiness $\theta$. Consider a Markov family $(\Omega$, $\mathcal{F}$, $(\eta_t)_{t \geq 0}$, $(\P_\mu)_{\mu \in \mathbf{N}})$, which is an unlabeled version of $(X_{k,t})_{k \in \N}$, $t \geq 0$. This construction is analogous to the one described in Lemma~\ref{lemma: correlated Brownian to unlabeled}. Let $\zeta$ denote a Pascal process with parameters $p$ and $\alpha$ that is proper, meaning it has the form \eqref{equation: zeta proper point process}. Recall that $\lambda$ denotes the Lebesgue measure on $\mathbb{R}$.

\begin{theorem}
    \label{Theorem: sticky Brownian motions}
    Suppose $\theta > 0$ and put $\alpha = \theta \lambda$. Then, for every $n \in \N$ and $0<p<1$, the infinite-dimensional Meixner polynomial of degree $n$ intertwines the dynamics of infinitely many uniform sticky Brownian motions with stickiness $\theta$ and their $n$-particle evolution. In other words,  
    \begin{align}
        \E_\zeta \sbra{\mathcal{M}^{p, \alpha}_n f_n(\eta_t)} = \mathcal{M}^{p, \alpha}_n P_t^{[n]} f_n(\zeta)
    \end{align}
    holds almost surely for all $t \geq 0$ and for all functions $f_n \in L^2(\lambda_n)$. 
\end{theorem}
Similar to Theorem~\ref{Theorem: correlated Brownian motions}, the proof is a direct consequence of consistency and reversibility for the $n$-particle evolutions.
\begin{proposition}
    \label{proposition: sticky n-particle reversible}
    For unlabeled uniform sticky Brownian motions $(\eta_t)_{t \geq 0}$ with stickiness $\theta$, the push-forward measure of $\lambda_n$ under the mapping $(x_1, \ldots, x_n) \to \delta_{x_1} + \ldots + \delta_{x_n}$ is reversible.  
\end{proposition}
To prove Proposition~\ref{proposition: sticky n-particle reversible}, we use \cite[Theorem 4.17]{brockingtonbethe}, which provides a reversible measure for $n$ ordered uniform sticky Brownian motions.

The following corollary is a novel result providing a family of reversible measures for a system of infinitely many sticky Brownian motions. The reversibility of $\rho_{p, \alpha}$, the distribution of the Pascal process, can be shown using the same arguments as in the proof of Corollary~\ref{corollary: correlated Brownian motions}. In this context, reversibility is defined in the analogue way as in \eqref{equation: definition reversible}.
\begin{corollary}
    \label{corollary: sticky Brownian motions}
    Let $\theta > 0$. Then, for each $0 < p < 1$, the distribution of the Pascal process with parameters $p$ and $\alpha = \theta \lambda$ is a reversible measure for $(\eta_t)_{t \geq 0}$, which is an infinite system of unlabeled uniform sticky Brownian motions with stickiness $\theta$. 
\end{corollary}

\section{Strategy for the Proof}
\label{section: strategy}

We prove the intertwining relations for correlated Brownian motions and sticky Brownian motions (Theorem~\ref{Theorem: correlated Brownian motions} and Theorem~\ref{Theorem: sticky Brownian motions}) by using the same techniques. First, in Section~\ref{section: strategy consistency}, we adapt the notation of consistency from e.g., \cite{carinci2021consistent} and \cite{intertwiningConsistent} to infinite particle systems. We then establish that the two models, correlated Brownian motions and sticky Brownian motions, are consistent, based only on the fact that the family of $n$-particle dynamics is a strongly consistent family according to Remark~\ref{remark: strongly consistent family}.

Next, in Section~\ref{section: abstract theorem}, we formulate a general theorem for both the Pascal and the Poisson case: if the push-forward of the factorial moment measures of the Poisson (or Pascal) process under the map $\iota_n : \R^n \to \mathbf{N}$, $(x_1, \ldots, x_n) \mapsto \delta_{x_1} + \ldots + \delta_{x_n}$ is reversible and the Markov processes are consistent, then we obtain the intertwining relations. The proof uses the explicit representation of Wiener-Itô integrals \eqref{equation: orthogonal polynomial: poisson case} and of infinite-dimensional Meixner polynomials \eqref{equation: infinite-dimensional Meixner polynomials}. 

Thus, we only need to verify the reversible measures for the finite systems. For correlated Brownian motions, this follows directly from the definition, while for sticky Brownian motions, we use \cite{brockingtonbethe}.

\subsection{Consistency}
\label{section: strategy consistency}

In \cite{carinci2021consistent} and \cite{intertwiningConsistent}, \emph{consistency} refers to the property that the removal of a particle uniformly at random commutes with the time-evolution of a process. This property can be formulated for particle systems with unlabeled particles. Specifically, if 
\begin{align}
    \label{equation: old definition consistency}
    \E_{\mu}\sbra{ \int F(\eta_t - \delta_x) \eta_t(\dd x)} = \int \E_{\mu - \delta_x} \sbra{F(\eta_t)} \mu(\dd x), \qquad \mu \in \Nfinite
\end{align}
holds for all measurable $F : \Nfinite \to [0, \infty)$, then the system is called consistent. In \eqref{equation: old definition consistency}, the left-hand side first evolves the system and then removes a particle uniformly at random. Conversely, the right-hand side of the equation first removes a particle uniformly at random from the initial configuration and then evolves the process from the resulting state.

In this article, we deliberately define consistency differently since \eqref{equation: old definition consistency} is not useful for extending to infinitely many particles. Equation~\eqref{equation: old definition consistency} lacks a connection between the dynamics of infinitely many particles and finitely many particles. More precisely, if one particle is removed from an infinite configuration, the total number of particles in the configuration remains infinite, whereas removing a particle from a finite configuration results in a finite number of remaining particles. However, our objective is to obtain intertwining relations that reduce the infinite number of particles to a finite number. According to \cite[Theorem 3.5, Remark 3.6.]{intertwiningConsistent}, our new definition below is a natural extension of the conventional definition of consistency, which coincides with the established notation of consistency of finite particle systems.

Let $(\Omega$, $\mathcal{F}$, $(\eta_t)_{t \geq 0}$, $(\P_\mu)_{\mu \in \mathbf{N}})$ be a Markov family such that $\eta_t$ is proper for each $t \geq 0$.

\begin{definition}
    A stochastic process $(\eta_t)_{t \geq 0}$ is called \emph{consistent} if 
    \begin{align}
        \label{Equation: New Definition Consistency}
        \notag
        &\E_{\mu} \sbra{ \int F(\delta_{x_1} + \ldots + \delta_{x_n}) \eta_t^{(n)}(\dd(x_1, \ldots, x_n))} \\
        &\hspace{4em} = \int \E_{\delta_{x_1} + \ldots + \delta_{x_n}} \sbra{F(\eta_t)} \mu^{(n)}(\dd(x_1, \ldots, x_n)), \qquad \mu \in \mathbf{N}
    \end{align}
    holds for all $t \geq 0$, measurable $F: \mathbf{N}_{<\infty} \to [0, \infty)$ and $n \in \N$. 
\end{definition}
We remark that, in general, integrability is not guaranteed on either the left-hand or the right-hand side in \eqref{Equation: New Definition Consistency}. In the case of non-integrability, we interpret the equation as $\infty = \infty$.

Interpreting \eqref{Equation: New Definition Consistency}, the left-hand side describes $n$ particles chosen uniformly from the evolved state of the process starting at $\mu$. The right-hand side, on the other hand, describes $n$ particles chosen uniformly from the initial state $\mu$ and then evolved under the process. 

\begin{remark}
    We only define consistency for processes in which each $\eta_t$, $t \geq 0$, is proper. This assumption is motivated by a technical subtlety: in general, the mapping $\mu \mapsto \int F(\delta_{x_1} + \ldots + \delta_{x_n}) \mu^{(n)}(\dd(x_1, \ldots, x_n))$ for $n \geq 2$, even for measurable $F$, is not measurable. 
    By requiring each $\eta_t$ to be proper, we ensure the measurability of $\omega \mapsto \int F(\delta_{x_1} + \ldots + \delta_{x_n}) \eta_t(\omega)^{(n)}(\dd(x_1, \ldots, x_n))$.
\end{remark}
Our definition can also be interpreted as self-intertwining property (see \cite[Equation~(3.6)]{intertwiningConsistent}) using the \emph{$K$-transform} introduced by Lenard~\cite{LenardI}, \cite{LenardII}. 
See also \cite{KunaHydrodynamicLimits} for $K$-transform intertwining relations in the context of \emph{free Kawasaki dynamics}.

We say that a process $(\eta_t)_{t \geq 0}$ is \emph{conservative} if it conserves the number of particles, meaning that $\eta_t(\R) = \mu(\R)$ holds $\P_\mu$-almost surely for all $\mu \in \mathbf{N}$ and $t \geq 0$.

\begin{proposition}
    \label{proposition: sticky / corellated brownian motion consistent, conservative}
    Both the unlabeled correlated Brownian motions with pairwise correlation $a$ and the unlabeled sticky Brownian motions with stickiness $\theta$ are consistent and conservative.
\end{proposition}
The proofs of both Lemma~\ref{lemma: correlated Brownian to unlabeled} and Proposition~\ref{proposition: sticky / corellated brownian motion consistent, conservative} rely solely on the strong consistency (picked up in Remark~\ref{remark: strongly consistent family}) of the family of underlying $n$-particle evolutions. This demonstrates that from any strongly consistent family, a consistent, unlabeled process can be constructed. Therefore, strong consistency is aptly named, as it is indeed a stronger property compared to consistency. For models of finitely many configurations, see also \cite[Section 3.1]{intertwiningConsistent}.

A direct consequence of the intertwining relation \eqref{Equation: New Definition Consistency}, which we take as the definition of consistency, is that $\pi_\lambda$, the distribution of the Poisson process with intensity measure $\lambda$, is an invariant measure for an infinite system of unlabeled correlated Brownian motions with pairwise correlation $a$. 
Indeed, we can use the fact that the moment problem for the Poisson process is uniquely solvable (as shown e.g. in \cite[Proposition 4.12]{LastPenroseLectures}). Therefore, it is enough to check that the factorial moment measures of a proper Poisson process $\zeta$, given by $\lambda^{\otimes n}$, and $\eta_t$ starting at $\zeta$ coincide. Since the push-forward measure of $\lambda^{\otimes n}$ under $\iota_n : (x_1, \ldots, x_n) \mapsto \delta_{x_1} + \ldots + \delta_{x_n}$ is invariant for $(\eta_t)_{t \geq 0}$, see Proposition~\ref{Proposition: correlated n-particle reversible}, we use consistency to obtain
\begin{align*}
    \int f_n \dx \lambda^{\otimes n} = \int P_t^{[n]} f_n \dx \lambda^{\otimes n} = \E \sbra{\int P_t^{[n]} f_n \dx \zeta^{(n)} } = \E \sbra{ \E_\zeta \sbra{\int f_n \dx \eta_t^{(n)}} }
\end{align*}
for measurable $f_n : E^n \to [0, \infty)$ and $t \geq 0$ which implies that $\pi_\lambda$ is indeed invariant for $(\eta_t)_{t \geq 0}$.

We emphasize that the argument for proving invariance of $\pi_\lambda$ is a general principle:
Let each $n$-th factorial moment measure of an infinite point process $\zeta$ be invariant for the $n$-particle dynamics of a consistent particle system. Assume that the factorial moment measures uniquely characterize the distribution of $\zeta$. Then, the distribution of $\zeta$ is invariant for the infinite dynamics.

Thus, using Proposition~\ref{proposition: sticky n-particle reversible}, we obtain that for each $0 < p < 1$, the distribution of the Pascal process with parameters $p$ and $\alpha = \theta \lambda$ is an invariant measure for an infinite system of unlabeled uniform sticky Brownian motions with stickiness $\theta$. 
Here, the unique solvability of the moment problem follows by the criterion presented in \cite[Proposition 4.12]{LastPenroseLectures}: Fix a bounded, measurable set $B \subset \R$. Since 
$\frac{1}{2^n n!} \alpha(B)^{(n)} \to 0$ 
as $n \to \infty$, there exists a $D \geq 1$ such that $\frac{1}{2^n n!} \alpha(B)^{(n)} \leq D$ for all $n \in \N$. Therefore, we estimate the factorial moment measure as follows:
\begin{align*}
    \bra{\frac{p}{1-p}}^n \lambda_n(B^n) &= \bra{\frac{p}{1-p}}^n \alpha(B)^{(n)} \leq \bra{2D \frac{p}{1-p}}^n n!.
\end{align*}

While the invariance follows easily by consistency, reversibility proves to be more challenging. In the following section, we conquer this issue by specifically exploiting the orthogonality of the intertwining relation in terms of infinite-dimensional orthogonal polynomials.

\subsection{Main Result: a Broader Perspective}
\label{section: abstract theorem}

The definition of consistency enables us to state the following theorem, which reveals the abstract framework after Theorem~\ref{Theorem: correlated Brownian motions} and Theorem~\ref{Theorem: sticky Brownian motions}. Let $(\Omega$, $\mathcal{F}$, $(\eta_t)_{t \geq 0}$, $(\P_\mu)_{\mu \in \mathbf{N}})$ be a Markov family (describing the evolution of infinitely many particles) such that $\eta_t$ is proper for each $t \geq 0$. Recall the mapping $\iota_n : \R^n \to \mathbf{N}$, $(x_1, \ldots, x_n) \mapsto \delta_{x_1} + \ldots + \delta_{x_n}$.
As a reminder, the set of configurations consisting of exactly $n \in \N_0$ particles is denoted by $\mathbf{N}_n := \set{\mu \in \mathbf{N} : \mu(\R) = n}$. 
Note that each symmetric function $f_n : \R^n \to \R$ can be identified with a function $F : \mathbf{N}_n \to \R$ through the relationship $f_n = F \circ \iota_n$. 
Suppose $(\eta_t)_{t \geq 0}$ is a conservative process. Then, for $n \in \N$, the $n$-particle semigroup $P_t^{[n]}$ can be recovered as follows:
\begin{align*}
    P_t^{[n]} f_n(x) := \E_{\iota_n(x)} \sbra{F(\eta_t)}, 
\end{align*}
where $x \in \R^n$, and $f_n : \R^n \to \R$ is a bounded (or non-negative), symmetric and measurable function. 
Define the operator $P_t^{[0]}$ as the identity operator on $\R$.
\begin{theorem}
    \label{Theorem: abstract theorem}
    Assume that $(\eta_t)_{t \geq 0}$ is consistent and conservative. 
    \begin{enumerate}[\normalfont(i)]
        \item \label{item: poisson} 
        Let $\zeta$ be a proper Poisson process with intensity measure $\lambda$. 
        Suppose that the push-forward measure of the Lebesgue measure on $\R^n$ under the mapping $\iota_n$ is reversible for $(\eta_t)_{t \geq 0}$ for each $n \in \N$. Then, the intertwining relation $\E_\zeta \sbra{I_n f_n(\eta_t)} = I_n P_t^{[n]} f_n(\zeta)$ holds almost surely and for all $t \geq 0$ and $f_n \in L^2_\mathrm{sym}(\lambda^{\otimes n})$, $n \in \N_0$. 
        \item \label{item: pascal} 
        Let $\zeta$ be a proper Pascal process with parameters $0 < p < 1$ and $\alpha$. 
        Suppose that the push-forward measure of $\lambda_n$ defined as \eqref{equation: definition lambda n} under the mapping $\iota_n$ is reversible for $(\eta_t)_{t \geq 0}$ for each $n \in \N$. Then, the intertwining relation $\E_\zeta \sbra{\mathcal{M}^{p, \alpha}_n f_n(\eta_t)} = \mathcal{M}^{p, \alpha}_n P_t^{[n]} f_n(\zeta)$ holds almost surely for all $t \geq 0$ and $f_n \in L^2_\mathrm{sym}(\lambda_n)$, $n \in \N_0$. 
    \end{enumerate}
\end{theorem}
We present a proof for Theorem~\ref{Theorem: abstract theorem} in Section~\ref{section: intertwining relations} below.
\begin{proof}[Proof of Theorem~\ref{Theorem: correlated Brownian motions} and Theorem~\ref{Theorem: sticky Brownian motions}]
    Proposition~\ref{proposition: sticky / corellated brownian motion consistent, conservative} shows that both the correlated and sticky Brownian motions, in their unlabeled version, are consistent and conservative. By using Proposition~\ref{Proposition: correlated n-particle reversible} (or Proposition~\ref{proposition: sticky n-particle reversible}), we apply Theorem~\ref{Theorem: abstract theorem} to obtain the intertwining relations for symmetric functions for both processes.

    If $f_n$ is a non-symmetric, square-integrable function, we apply Theorem~\ref{Theorem: abstract theorem} to its symmetrization $\widetilde{f_n}$, and then use 
    \begin{align*}
        P_t^{[n]}{\widetilde{f_n}} = \widetilde{P_t^{[n]} f_n},
    \end{align*}
    which follows from strong consistency, as defined in \eqref{equation: strong consistency}, together with $I_n \widetilde{f_n} = I_n f_n$ (or $\mathcal{M}^{p, \alpha}_n \widetilde{f_n} = \mathcal{M}^{p, \alpha}_n f_n$). 
\end{proof}

The proof of Corollary~\ref{corollary: correlated Brownian motions} (or Corollary~\ref{corollary: sticky Brownian motions}) also works in the abstract framework. Recall that $\pi_\lambda$ denotes the distribution of the Poisson process with intensity measure $\lambda$, and $\rho_{p, \alpha}$ denotes the distribution of the Pascal process with parameters $p$ and $\alpha$. 

\begin{corollary}
    \label{corollary: abstract}
    Under the assumptions of Theorem~\ref{Theorem: abstract theorem}, if we assume the condition stated in~{\normalfont(\ref{item: poisson})}, then $\pi_\lambda$ is reversible for $(\eta_t)_{t \geq 0}$, whereas if we assume the condition stated in~{\normalfont(\ref{item: pascal})}, then $\rho_{p, \alpha}$ is reversible for $(\eta_t)_{t \geq 0}$. 
\end{corollary}

\begin{remark}
    Theorem~\ref{Theorem: abstract theorem} and Corollary~\ref{corollary: abstract} are formulated here for Markov processes with particles on the real line. However, both results hold true in a much more general setting. For instance, if $(E, \mathcal{E})$ is a \emph{Borel space} (see \cite{LastPenroseLectures}), instead of the Lebesgue measure, an arbitrary $\sigma$-finite measure can be chosen for both the Poisson and Pascal case. 
\end{remark}

\begin{remark}
The condition of reversibility in Theorem~\ref{Theorem: abstract theorem} can be relaxed by demanding only the condition \eqref{equation: condition poisson} below. Moreover, if this equation holds pointwise, rather than almost everywhere, we obtain the intertwining relation \eqref{equation: intertwining Poisson} pointwise as well. However, to ensure the existence of integrals, it may be necessary to reduce the set of configurations. We can consider this idea for systems of \emph{independent particles} on a Borel space: if the one-particle dynamics admits only an invariant, $\sigma$-finite measure, then \eqref{equation: condition poisson} follows pointwise.

By being more rigorous, we obtain the following statement. Consider a Markov semigroup $(p_t)_{t \geq 0}$ on a Borel space $(E, \mathcal{E})$ with an invariant $\sigma$-finite measure $\lambda$. Let $(\Omega$, $\mathcal{F}$, $(\eta_t)_{t \geq 0}$, $(\P_\mu)_{\mu \in \mathbf{N}})$ be the Markov family of unlabeled particles, where each particle evolves independently according to the single particle dynamics $(p_t)_{t \geq 0}$. This family can be constructed using the arguments presented in Lemma~\ref{lemma: correlated Brownian to unlabeled}.

Let $t \geq 0$ be fixed, and assume that $\mu$ is locally finite, meaning that $\mu(A) < \infty$ for all $A \in \mathcal{E}$ with $\lambda(A) < \infty$, as well as the measure $\mu p_t$, defined by $\mu p_t(B) := \int p_t \one_B \dx \mu$. We also consider a measurable and bounded function $f_n : E^n \to \R$, $n \in \N$, such that $\lambda^{\otimes n}[f_n \neq 0] < \infty$.
Under these conditions, we have the following intertwining relation:
\begin{align*}
    \E_\mu \sbra{I_n f_n(\eta_t)} = I_n P_t^{[n]} f_n(\mu)
\end{align*}
where $I_n$ denotes the multiple Wiener-Itô integral of degree $n$, constructed with the invariant measure $\lambda$.
\end{remark}

\section{Proofs}
\label{section: proofs}
\subsection{Unlabeled Dynamics and Consistency}
\label{section: proof consistent}

First, we prove Lemma~\ref{lemma: correlated Brownian to unlabeled} and Proposition~\ref{proposition: sticky / corellated brownian motion consistent, conservative}. In the following proof, we use a construction of the unlabeled dynamics that combines two principles. Firstly, we combine the dynamics of different numbers of particles, choosing which dynamics to follow based on the initial configuration's particle number. Secondly, we map the labeled notation $x = (x_k)_{k=1}^n \in \R^n$ to $\iota_n(x) = \sum_{k=1}^n \delta_{x_k}$. These arguments are standard: the first one follows easily, and the second can be interpreted as a Markov mapping theorem. However, since this construction is essential for the intertwining relations, we present it in detail for the reader's benefit.

\begin{proof}[Proof of Lemma~\ref{lemma: correlated Brownian to unlabeled}]
    For each $n \in \N \cup \set{\infty}$, let $(\Omega^n, \mathcal{F}^n, (Z_t^n)_{t \geq 0}$, $(\P_x^n)_{x \in \R^n})$ be the Markov family that describes the dynamics of $n$ particles. The family consists of a measurable space $(\Omega^n, \mathcal{F}^n)$, measurable mappings $Z_t^n : \Omega \to \R^n$, $t \geq 0$, and probability measures $\P_x^n$ on $(\Omega^n, \mathcal{F}^n)$ with corresponding expected values denoted by $\E_x^n$. The Markov property is satisfied with respect to the natural filtration $\mathcal{F}_t^n := \sigma(Z_s^n : 0 \leq s \leq t)$. 
    \begin{itemize}
        \item We define $\Omega := \set{0} \cup \bigcup_{n \in \N \cup \set{\infty}} \set{n} \times \Omega_n$, and equip it with the $\sigma$-algebra $\mathcal{F}$ generated by the sets $\set{0}$ and $\set{n} \times A_n$, $A_n \in \mathcal{F}^n$, $n \in \N_0 \cup \set{\infty}$
        \item Put $\eta_t(n, \omega^n) := \iota_n(Z_t^n(\omega^n))$ for $n \in \N \cup \set{\infty}$, $\omega^n \in \Omega^n$ and $\eta_t(0) := 0$. 
        \item For every $\mu \in \mathbf{N}$, we choose a fixed $z(\mu) \in \R^n$ such that $\iota_n(z(\mu)) = \mu$, where $n = \mu(E)$. We then define $\P_\mu$ to be the push-forward measure of $\P_{z(\mu)}^n$ under the mapping $\omega^n \mapsto (n, \omega^n)$.
    \end{itemize}
    The distribution $\P_\mu$ relies on how the components of $z(\mu)$ are permuted, but according to the strong consistency property, see Remark~\ref{remark: strongly consistent family}, the distribution of $(\eta_t)_{t \geq 0}$ under $\P_\mu$ remains unchanged regardless of the choice of permutation. In particular, this distribution is equal to the distribution of $(\iota_n(Z_t^n))_{t \geq 0}$ under $\P_x$, for all $x \in \R^n$ that satisfy $\iota_n(x) = \mu$.

    We now show that for $\mathcal{F}_t = \sigma(\eta_s: 0 \leq s \leq t)$ the Markov property
    \begin{align*}
        \P_\mu[\eta_{t+s} \in A \mid \mathcal{F}_s] = \P_{\eta_s} \sbra{ \eta_t \in A} \qquad \P_\mu \text{-almost surely for all } \mu \in \mathbf{N}, A \in \mathcal{N}, t, s \geq 0
    \end{align*}
    holds by applying the Markov property of the process $(Z_t^n)_{t \geq 0}$. 
    Indeed, fix $A \in \mathcal{N}$, $s, t \geq 0$, $\mu = \sum_{k=1}^n \delta_{x_k} \in \mathbf{N}$, $x = (x_k)_{k=1}^n$, $n \in \N \cup \set{\infty}$ and $B \in \mathcal{F}_s$. We use the abbreviations $\iota = \iota_n$, $Z_t = Z_t^n$, $\E_x = \E_x^n$. Then, using the definition of $\mathcal{F}_s$, we find a measurable set $C \subset \set{h : h: [0, t] \to \mathbf{N}}$ such that $\one_{B} = \one_C((\eta_u)_{0 \leq u \leq s})$. Therefore, 
    \begin{align*}
        \E_\mu \sbra{\one_A(\eta_{t+s}) \one_{B}} &= \E_x \sbra{\one_A(\iota(Z_{t+s}) ) \one_C \bra{\bra{\iota(Z_u)}_{0 \leq u \leq s}}} \\
        &= \E_x \sbra{\E_{Z_s} \sbra{ \one_A(\iota(Z_t))} \one_C\bra{\bra{\iota(Z_u)}_{0 \leq u \leq s}}} = \E_\mu \sbra{\E_{\eta_s} \sbra{ \one_A(\eta_t)} \one_{B}}
    \end{align*}
    holds since $\one_C \bra{\bra{\iota_n(Z_u)}_{0 \leq u \leq s}}$ is $\mathcal{F}_s^n$-measurable. The remaining properties can be easily obtained.
\end{proof}
The proof of Proposition~\ref{proposition: sticky / corellated brownian motion consistent, conservative} relies solely on strong consistency (see Remark~\ref{remark: strongly consistent family}) and follows by linearity and by the definition of the factorial measure.

\begin{proof}[Proof of Proposition~\ref{proposition: sticky / corellated brownian motion consistent, conservative}]
By definition, $(\eta_t)_{t \geq 0}$ is conservative, so only consistency requires a proof.
Let $n \in \N \cup \set{\infty}$ be fixed and consider $\mu = \sum_{k=1}^n \delta_{x_k} \in \mathbf{N}$, where $x = (x_k)_{k=1}^n$. Let $(\Omega^n, \mathcal{F}^n, (Z_t^n)_{t \geq 0}$, $(\P_x^n)_{x \in \R^n})$, $Z_t^n = (Z_{k,t}^n)_{k =1}^n, t \geq 0$ denote the Markov family associated with the $n$-particle dynamics as described in the proof of Lemma~\ref{lemma: correlated Brownian to unlabeled}. By applying \eqref{equation: factorial measure equals sum} twice, and using strong consistency and linearity, we obtain consistency:
\begin{align*} 
    \E_\mu \sbra{\int F(\delta_{y_1} + \ldots + \delta_{y_l}) \eta_t^{(l)} (\dd(y_1, \ldots, y_l))} &= \sum_{\substack{i_1, \ldots, i_l = 1 \\ \text{pairwise different}}}^n \E_x^n \sbra{ F(\delta_{Z_{i_1,t}^n} + \ldots, \delta_{Z_{i_l, t}^n}) }  \\
    = \sum_{\substack{i_1, \ldots, i_l = 1 \\ \text{pairwise different}}}^n \E_{\delta_{x_{i_1}} + \ldots + \delta_{x_{i_l}}} \sbra{F(\eta_t)} &= \int \E_{\delta_{y_1} + \ldots + \delta_{y_l}} \sbra{F(\eta_t)} \mu^{(l)}(\dd (y_1, \ldots, y_l)). \qedhere 
\end{align*}
\end{proof}

\subsection{Intertwining Relations}
\label{section: intertwining relations}
To begin, we present a proof for part~{\normalfont(\ref{item: poisson})} of Theorem~\ref{Theorem: abstract theorem}. Recall the mapping $\iota_n : \R^n \to \mathbf{N}$, $(x_1, \ldots, x_n) \mapsto \delta_{x_1} + \ldots + \delta_{x_n}$. Note that 
\begin{align}
    \label{equation: tensor and factorial measure} 
    \bra{\widetilde{u \otimes \one}}(x_1, \ldots, x_{n+1}) 
    &= \frac{1}{(n+1)!} \int u \dx \bra{\delta_{x_1} + \ldots + \delta_{x_{n+1}}}^{(n)}
\end{align}
for $x_1, \ldots, x_{n+1} \in \R$ and measurable $u: \R^n \to \R$ using the notation $\widetilde{u \otimes \one}$ for the symmetrization of the function $(u \otimes \one)(x_1, \ldots, x_{n+1}) := u(x_1, \ldots, x_n)$. 
With this notation, the consistency property for a finite number of particles can be expressed in terms of $P_t^{[n]}$ by
\begin{align}
    \label{equation: consistency in terms of Pt}
    \notag
    P_t^{[n+1]} \bra{\widetilde{\varphi \otimes \one}}(x_1, \ldots, x_{n+1}) &= \frac{1}{(n+1)!} \int P_t^{[n]} \varphi \dx \bra{\delta_{x_1} + \ldots + \delta_{x_{n+1}}}^{(n)} \\
    &=  \bra{\widetilde{(P_t^{[n]} \varphi) \otimes \one}}(x_1, \ldots, x_{n+1})
\end{align}
for bounded, symmetric $\varphi : \R^n \to \R$. 
To obtain this expression, we use \eqref{equation: tensor and factorial measure} for both $\varphi$ and $P_t^{[n]} \varphi$, and apply the consistency condition \eqref{Equation: New Definition Consistency}. 

The crucial steps in the following proof are as follows: first, we establish the intertwining relation for functions in the smaller space $\mathcal{C}_n$ using the explicit formulas for the orthogonal polynomials \eqref{equation: orthogonal polynomial: poisson case} and \eqref{equation: infinite-dimensional Meixner polynomials}. Next, we extend this relation to all functions in $L^2(\lambda^{\otimes n})$ using an approximation argument. To accomplish this, we employ the moment problem to demonstrate that $\pi_\lambda$ is invariant.

\begin{proof}[Proof of Theorem~\ref{Theorem: abstract theorem}~{\normalfont(\ref{item: poisson})}]
    We claim the following equation: for all $t \geq 0$, $l \in \N_0$ and measurable $F: \mathbf{N}_{l+1} \to [0, \infty)$, 
    \begin{align}
        \label{equation: condition poisson}
        \int \E_{\delta_{z_1} + \ldots + \delta_{z_l} + \delta_y} \sbra{ F(\eta_t)}\lambda(\dd y) = \int \E_{\delta_{z_1} + \ldots + \delta_{z_l}} \sbra{F(\eta_t + \delta_y)} \lambda(\dd y)
    \end{align}
    holds for $z = (z_1, \ldots, z_l) \in \R^l$ $\lambda^{\otimes l}$-almost everywhere. The case where $l=0$ reads as follows: $\int \E_{\delta_y} \sbra{ F(\eta_t)} \lambda(\dd y) = \int F(\delta_y) \lambda(\dd y)$.

    To prove \eqref{equation: condition poisson}, we multiply both the right-hand side and the left-hand side of the equation by an arbitrary measurable function $\varphi : \R^l \to [0, \infty)$ and integrate with respect to $\lambda^{\otimes l}$. Since both the right-hand side and the left-hand side of \eqref{equation: condition poisson} are symmetric in $(z_1, \ldots, z_l)$, it is sufficient to integrate with symmetric functions $\varphi$. Let $t \geq 0$, $l \in \N$. Using reversibility, we obtain
    \begin{align}
        \label{equation: right side}
        &\int \varphi(x_1, \ldots, x_l) \int \E_{\delta_{x_1} + \ldots + \delta_{x_l} + \delta_y} \sbra{F(\eta_t)} \lambda(\dd y) \lambda^{\otimes l}(\dd (x_1, \ldots, x_l)) \notag \\
        &= \int \widetilde{\varphi \otimes \one}(x_1, \ldots, x_l, y) \E_{\delta_{x_1} + \ldots + \delta_{x_l} + \delta_y} \sbra{F(\eta_t)} \lambda^{\otimes (l+1)} (\dd (x_1, \ldots, x_l,y)) \notag \\
        &= \int P_t^{[l+1]} \bra{\widetilde{\varphi \otimes \one}}(x_1, \ldots, x_l,y) F(\delta_{x_1} + \ldots + \delta_{x_l} + \delta_y) \lambda^{\otimes (l+1)}(\dd (x_1, \ldots, x_l, y )).
    \end{align}
    Applying \eqref{equation: consistency in terms of Pt} and using reversibility once again, \eqref{equation: right side} can be transformed into 
    \begin{align*}
        &\int \bra{\widetilde{(P_t^{[l]} \varphi) \otimes \one}}(x_1, \ldots, x_l,y)  F(\delta_{x_1} + \ldots + \delta_{x_l} + \delta_y) \lambda^{\otimes (l+1)}(\dd (x_1, \ldots, x_l, y)) \\
        &\hspace{4em}= \int P_t^{[l]} \varphi (x_1, \ldots, x_l)  \int F(\delta_{x_1} + \ldots + \delta_{x_l} + \delta_y) \lambda(\dd y) \lambda^{\otimes l}(\dd (x_1, \ldots, x_l)) \\
        &\hspace{4em}= \int \varphi (x_1, \ldots, x_l)  \int \E_{\delta_{x_1} + \ldots + \delta_{x_l}} \sbra{F(\eta_t + \delta_y)} \lambda(\dd y) \lambda^{\otimes l}(\dd (x_1, \ldots, x_l)).
    \end{align*}
    which implies \eqref{equation: condition poisson}. 

    Let $\zeta$ be proper Poisson process with intensity measure $\lambda$. 
    Equation~\eqref{equation: condition poisson} enables us to prove the intertwining relation \eqref{equation: intertwining Poisson}. Let $f_n \in \mathcal{C}_n$ be symmetric and select $F : \mathbf{N}_n \to \R$ such that $f_n = F \circ \iota_n$. Consequently, by using \eqref{equation: orthogonal polynomial: poisson case} and consistency, we arrive at the following 
    \begin{align*}
        &\E_\zeta \sbra{I_n f_n(\eta_t)} \\
        &\hspace{4em}= \sum_{k=0}^n \binom{n}{k} (-1)^{n-k} \E_\zeta \bigg [\iint f_n(x_1, \ldots, x_n) \lambda^{\otimes (n-k)}(\dd(x_{k+1}, \ldots, x_n)) \eta_t^{(k)}(\dd(x_1, \ldots, x_k))  \bigg ] \\
        &\hspace{4em}= \sum_{k=0}^n \binom{n}{k} (-1)^{n-k} \iint \E_{\delta_{x_1} + \ldots + \delta_{x_k}} \sbra{F(\eta_t + \delta_{x_{k+1}} + \ldots + \delta_{x_n})} \\
        &\hspace{16em}  \lambda^{\otimes (n-k)}(\dd(x_{k+1}, \ldots, x_n)) \zeta^{(k)}(\dd (x_1, \ldots, x_k)).
    \end{align*}
    Using \eqref{equation: condition poisson} repeatedly for $n-k$, we obtain
    \begin{align*}
        &\int \E_{\delta_{x_1} + \ldots + \delta_{x_k}} \sbra{F(\eta_t + \delta_{x_{k+1}} + \ldots + \delta_{x_n}) } \lambda^{\otimes (n-k)}(\dd(x_{k+1}, \ldots, x_n)) \\
        &\hspace{4em}= \int \E_{\delta_{x_1} + \ldots + \delta_{x_n}} \sbra{F(\eta_t)} \lambda^{\otimes (n-k)}(\dd(y_{x+1}, \ldots, x_n)) \\
        &\hspace{4em}= \int P_t^{[n]} f_n(x_1, \ldots, x_n) \lambda^{\otimes (n-k)}(\dd(x_{k+1}, \ldots, x_n))
    \end{align*}
    for $\lambda^{\otimes k}$-almost all $(y_1, \ldots, y_k)$. Since $\lambda^{\otimes k}$ is the $k$-th factorial moment measure of the Poisson process with intensity measure $\lambda$, integrating with respect to the $k$-th factorial measure of $\zeta$ is well-defined, resulting in
    \begin{align*}
        &\E_\zeta \sbra{I_n f_n(\eta_t)} \\
        &\hspace{4em}= \sum_{k=0}^n \binom{n}{k} (-1)^{n-k} \iint P_t^{[n]} f_n(x_1, \ldots, x_n) \lambda^{\otimes (n-k)}(\dd(x_{k+1}, \ldots, x_n)) \zeta^{\otimes k}(\dd (x_1, \ldots, x_k)) \\
        &\hspace{4em}= I_n P_t^{[n]} f_n(\zeta)
    \end{align*}
    almost surely.

    As a next step, we extend \eqref{equation: intertwining Poisson} to include $f_n \in L^2_{\mathrm{sym}}(\lambda^{\otimes n})$.
    In Section~\ref{section: strategy consistency}, we observed that under the assumptions of Theorem~\ref{Theorem: abstract theorem}~{\normalfont(\ref{item: poisson})}, the measure $\pi_\lambda$ is invariant for $(\eta_t)_{t \geq 0}$.
    Thus, the expected value $\E_\zeta \sbra{F(\eta_t)}$ is well-defined for equivalence classes under $\pi_\lambda$. Moreover, we have 
    \begin{align}
        \label{equation: inequality E zeta}
        \norm{\E_\zeta \sbra{F(\eta_t)}}_{L^2} \leq \norm{F}_{L^2(\pi_\lambda)}, \qquad F \in L^2(\pi_\lambda)
    \end{align}
    where $L^2 := L^2(\Omega, \mathcal{F}, \P)$ and $(\Omega, \mathcal{F}, \P)$ denotes the probability space on which $\zeta$ is defined.
    There exists a sequence $f_n^k \in \mathcal{C}_n$, since $\mathcal{C}_n$ is dense in $L^2(\lambda^{\otimes n})$,  such that $f_n^{k} \to f_n$ as $k \to \infty$. Since $f_n$ is symmetric, we can choose each $f_n^k$ to be symmetric as well. 
    Using \eqref{equation: inequality E zeta} and the orthogonality relation \eqref{equation: orthogonality relation poisson}, we obtain
    \begin{align*}
        &\norm{\E_\zeta \sbra{I_n f_n(\eta_t)} - I_n P_t^{[n]} f_n(\zeta)}_{L^2} \\
        &\hspace{4em}\leq \norm{\E_\zeta \sbra{I_n f_n(\eta_t) - I_n f_n^k(\eta_t)}}_{L^2} +  \norm{I_n P_t^{[n]} f_n(\zeta) - I_n P_t^{[n]} f_n^k(\zeta)}_{L^2 } \\ 
        &\hspace{4em}\leq \sqrt{n} \norm{f_n - f_n^k}_{L^2(\lambda^{\otimes n})} + \sqrt{n} \norm{P_t^{[n]} f_n - P_t^{[n]} f_n^k}_{L^2(\lambda^{\otimes n})} \\ 
        &\hspace{4em}\leq 2 \sqrt{n} \norm{f_n -  f_n^k}_{L^2(\lambda^{\otimes n})} \to 0, \qquad k \to \infty. 
    \end{align*}
    In the third inequality, we exploit the property that $P_t^{[n]}$ is a contraction on $L^2_{\mathrm{sym}}(\lambda^{\otimes n})$, a direct consequence of the invariance of the push-forward measure of $\lambda^{\otimes n}$ under $\iota_n$. Therefore, $\E_\zeta \sbra{I_n f_n(\eta_t)} = I_n P_t^{[n]} f_n(\zeta)$ almost surely. 
\end{proof}

Recall $(a)^{(k)} := a (a+1) \cdots (a+k-1)$, $(a)^{(0)} := 1$, the \emph{rising factorial}, and $(a)_k := a (a-1) \cdots (a-k+1)$, $(a)_0 := 1$, the \emph{falling factorial}.

\begin{proof}[Proof of Theorem~\ref{Theorem: abstract theorem}~{\normalfont(\ref{item: pascal})}]
The proof is analogous to the one for the Poisson case up to minor changes: Equation~\eqref{equation: condition poisson} has an analogous form for the Pascal case, given by
\begin{align*}
    \int \E_{\delta_{z_1} + \ldots + \delta_{z_l} + \delta_y} \sbra{F(\eta_t)} (\delta_{z_1} + \ldots + \delta_{z_l} + \alpha)(\dd y)  = \E_{\delta_{z_1} + \ldots + \delta_{z_l}} \sbra{ \int F(\eta_t + \delta_y) \: (\eta_t + \alpha)(\dd y) }
\end{align*}
for $t \geq 0$, $l \in \N_0$, measurable $F: \mathbf{N}_{l+1} \to [0, \infty)$ and $z = (z_1, \ldots, z_l) \in \R^l$ $\lambda_l$-almost everywhere.

\end{proof}

\begin{proof}[Proof of Proposition~\ref{proposition: factorial moment measure Pascal process}]
    Let $\zeta$ be a Pascal process with parameters $p$ and $\alpha$ and fix $A = A_1^{d_1} \times \cdots \times A_N^{d_N}$ where $A_1, \ldots, A_N \subset \R$ are disjoint measurable sets and $d_1 + \ldots + d_N = n$. Firstly, note that 
    (see e.g. \cite[Equation~(3.4)]{intertwiningConsistent}) 
    \begin{align}
        \label{equation: indicator falling factorial}
        \int \one_A \dx \mu^{(n)} = (\mu(A_1))_{d_1} \cdots (\mu(A_N))_{d_N}.
    \end{align}
    Secondly, Lemma~\ref{lemma: Meixner 2} below implies
    \begin{align}
        \label{equation: indicator lambda n}
        \int \one_A \dx \lambda_n = (\alpha(A_1))^{(d_1)} \cdots (\alpha(A_N))^{(d_N)}.
    \end{align}
    By combining \eqref{equation: indicator falling factorial} and \eqref{equation: indicator lambda n}, and using the fact that the $k$-th factorial moment of the negative Binomial distribution with parameters $p$ and $a$ is $\bra{\frac{p}{1-p}} (a)^{(k)}$ we obtain $\E \sbra{\int \one_A \dx \mu^{(n)}} = \bra{\frac{p}{1-p}}^n \lambda_n(A)$. This equation can be extended to all measurable subsets of $\R^n$ using standard measure-theoretic arguments.
\end{proof}

\subsection{Infinite-Dimensional Meixner Polynomials: an Explicit Formula}
\label{section: explicit formula Meixner}

In this section, we prove Proposition~\ref{proposition: explicit formula infinite-dimensional Meixner}. Our strategy for the proof is as follows: on the one hand, as mentioned in Remark~\ref{remark: univariate meixner}, $\mathcal{M}^{p, \alpha}_n f_n$ defined by \eqref{equation: infinite-dimensional Meixner polynomials} for specific $f_n$ reduces to a product of univariate Meixner polynomials. We prove this property in Proposition~\ref{lemma: explicit meixner product formula} below, for which we state two preliminary lemmas. On the other hand, it is well-known (see \cite[Proposition 5.3]{intertwiningConsistent}, see also \cite[Lemma 3.1]{LYTVYNOV2003118}) that the infinite-dimensional orthogonal polynomials defined by the right-hand side of \eqref{equation: explicit formula infinite-dimensional Meixner} have the same product structure for this $f_n$. The equality then is obtained through standard measure-theoretical arguments.

Fix a partition of measurable sets $B_1, \ldots, B_N$ of $\R$, $n < m$, $z_1, \ldots, z_n \in \R$ and put $c_k := \bra{\delta_{z_1} + \ldots + \delta_{z_n}}(B_k)$ for $k \in \set{1, \ldots, N}$. 

\begin{lemma}
    \label{lemma: Meixner 1}
    Let $i_{n+1}, \ldots, i_{m} \in \set{1, \ldots, N}$ and put $e_k := \sum_{l=n+1}^m \one_{\set{i_l= k}}$ for $k \in \set{1, \ldots, N}$. Then, 
    \begin{align*}
         \int  \one_{B_{i_{n+1}} \times \ldots \times B_{i_{m}}}(y_{n+1}, \ldots, y_{m})  \: \kappa_{m,n} (z_1, \ldots, z_n, \dd( y_{n+1}, \ldots, y_{m})) = \prod_{k=1}^N \bra{\alpha(B_k) + c_k}^{(e_k)}
    \end{align*}
    holds true. 
\end{lemma}
Thereby, we put $\infty^{(k)} := \infty$ for $k \geq 1$ and $\infty^{(0)} := 0$.
\begin{proof}
    We prove the equation by induction over $m$. For $m=n+1$ the statement is a direct consequence of the definition of $\kappa_{n+1, n}$. Assume that the statement is true for some fixed $m > n$. Let
    \begin{align*}
        \sum_{l=n+1}^{m+1} \one_{\set{i_l = k}} = e_k + \one_{\set{k = i_{m+1}}}
    \end{align*}
    for an arbitrary $i_{m+1} \in \set{1, \ldots, N}$. Then, by \eqref{equation: definition kappa}, 
    \begin{align*}
        &\int  \one_{B_{i_{n+1}} \times \ldots B_{i_{m+1}}}(y_{n+1}, \ldots, y_{m+1}) \kappa_{m+1,n}(z_1, \ldots, z_n, \dd (y_{n+1}, \ldots, y_{m+1}))\\
        &\hspace{4em}= \int  \one_{B_{i_{n+1}} \times \ldots \times B_{i_{m}}}(y_{n+1}, \ldots, y_{m})  \kappa_{m+1, m}(z_1, \ldots, z_n, y_{n+1}, \ldots, y_{m}, B_{i_{m+1}}) \\
        &\hspace{21em} \kappa_{m,n}(z_1, \ldots, z_n, \dd (y_{n+1}, \ldots, y_{m})) \\
        &\hspace{4em}= (\alpha(B_{i_{m+1}}) + c_{i_{m+1}} + e_{i_{m+1}} ) \prod_{k=1}^N \bra{\alpha(B_k) + c_k}^{(e_k)} =\prod_{k=1}^N \bra{\alpha(B_k) + c_k}^{(e_k + \one_{\set{k = i_{m+1}}})}. \qedhere
    \end{align*}
\end{proof}

\begin{lemma}
\label{lemma: Meixner 2}
Let $d_1, \ldots, d_N \in \N_0$ with $d_1 + \ldots + d_N = m$. Then,
\begin{align}
    \label{equation: meixner 2}
    \notag
    &\int \widetilde{\one}_{B_1^{d_1} \times \cdots \times B_N^{d_N}} (z_1, \ldots, z_n, y_{n+1}, \ldots, y_{m}) \: \kappa_{m,n}(z_1, \ldots, z_n, \dd( y_{n+1}, \ldots, y_{m}))  \\
    &\hspace{4em}=  \frac{1}{(m)_n}  \prod_{k=1}^N  (d_k)_{c_k} \bra{\alpha(B_k) + c_k}^{(d_k - c_k)}
\end{align}
holds where $\widetilde{\one}_{B_1^{d_1} \times \cdots \times B_N^{d_N}}$ denotes the symmetrization of $\one_{B_1^{d_1} \times \cdots \times B_N^{d_N}}$. 
\end{lemma}
In case there exists a $k$ for which $d_k < c_k$, $(d_k)_{c_k}$ becomes zero, leading to the right-hand side of \eqref{equation: meixner 2} being zero. Moreover, if there exists a $k$ for which $d_k > c_k$ and $\alpha(B_k) = \infty$, the right-hand side of \eqref{equation: meixner 2} is equal to infinity.

\begin{proof}
We decompose the the integral into 
\begin{align}
    \label{equation: decomposition integral}
    \notag
   &\int \widetilde{\one}_{B_1^{d_1} \times \cdots \times B_N^{d_N}}(z_1, \ldots, z_n, y_{n+1}, \ldots, y_{m}) \kappa_{m,n}(z_1, \ldots, z_n, \dd( y_{n+1}, \ldots, y_{m})) \\
   \notag
   &\hspace{2em}= \sum_{i_{n+1}, \ldots, i_{m} \in \set{1, \ldots, N}} \int \one_{B_{i_{n+1}} \times \ldots \times B_{i_{m}}}(y_{n+1}, \ldots, y_{m}) \widetilde{\one}_{B_1^{d_1} \times \cdots \times B_N^{d_N}}(z_1, \ldots, z_n, y_{n+1}, \ldots, y_{m}) \\
   &\hspace{16em} \kappa_{m,n}(z_1,\ldots,z_n, \dd( y_{n+1}, \ldots, y_{m})).
\end{align}
If $x_1, \ldots, x_m \in \R$ are given such that $(\delta_{x_1} + \ldots + \delta_{x_m})(B_k) = d_k$ for all $k \in \set{1, \ldots, N}$, then the symmetrization $\widetilde{\one}_{B_1^{d_1} \times \cdots \times B_N^{d_N}} (x_1, \ldots, x_m)$ is equal to $\frac{d_1! \cdots d_N!}{m!}$. Otherwise, it is equal to zero. Thus, $\widetilde{\one}_{B_1^{d_1} \times \cdots \times B_N^{d_N}}(z_1, \ldots, z_n, y_{n+1}, \ldots, y_{m}) > 0$ is equivalent to $\sum_{l=n+1}^m \one_{\set{i_l = k}} + c_k = d_k$. 
By applying Lemma~\ref{lemma: Meixner 1} to \eqref{equation: decomposition integral}, we obtain:
\begin{align*}
    &\frac{d_1! \cdots d_N!}{m!} \sum_{\substack{i_{n+1}, \ldots, i_{m} \in \set{1, \ldots, N} \\  \sum_{l=n+1}^m \one_{\set{i_l= k}} + c_k = d_k }} \int  \one_{B_{i_{n+1}} \times \ldots \times B_{i_{m}}}(y_{n+1}, \ldots, y_{m})  \kappa_{m,n}(z_1, \ldots, z_n, \dd( y_{n+1}, \ldots, y_{m})) \\
   &=  \frac{d_1! \cdots d_N!}{m!} \sum_{\substack{i_{n+1}, \ldots, i_{m} \in \set{1, \ldots, N} \\  \sum_{l=n+1}^m \one_{\set{i_l = k}} + c_k = d_k }} \prod_{k=1}^N \bra{\alpha(B_k) + c_k}^{(d_k - c_k)}. 
\end{align*}
Finally, using the identity
\begin{align*}
    \sum_{\substack{i_{n+1}, \ldots, i_{m} \in \set{1, \ldots, N} \\  \sum_{l=n+1}^m \one_{\set{i_l= k}} + c_k = d_k }} 1 = (m-n)!\prod_{k=1}^N \one_{\set{d_k \geq c_k}} \frac{1}{(d_k - c_k)!},
\end{align*}
we conclude the proof.
\end{proof}
Recall that the monic, univariate Meixner polynomials are given by \eqref{equation monic Meixner}.
\begin{lemma}
\label{lemma: explicit meixner product formula}
Let $d_1, \ldots, d_N \in \N_0$ with $d_1 + \ldots + d_N = m$ and $\mu \in \mathbf{N}$. Assume $\mu(B_k) < \infty$ and $0 < \alpha(B_k) < \infty$ for all $k$ with $d_k > 0$. Then, the equation
\begin{align*}
    \mathcal{M}^{p, \alpha}_m \one_{B_1^{d_1} \times \cdots \times B_N^{d_N}}(\mu) =  \prod_{k=1}^N \mathscr{M}^{p,\alpha(A_k)}_{d_k}\bra{\mu(B_k)}
\end{align*}
is satisfied. 
\begin{proof}
    Let $b_k := \mu(B_k) \in \N_0$. By applying Lemma~\ref{lemma: Meixner 1} and the definition of the factorial measure $\mu^{(n)}$, we obtain
    \begin{align*}
        &\iint \widetilde{\one}_{B_1^{d_1} \times \cdots \times B_N^{d_N}} (z_1, \ldots, z_n, y_{n+1}, \ldots, y_{m}) \: \kappa_{m,n}(z_1, \ldots, z_n, \dd (y_{n+1}, \ldots, y_{m})) \mu^{(n)}(\dd (z_1, \ldots, z_n)) \\
        &\hspace{4em}= \frac{n!}{(m)_n}  \sum_{\substack{c_1, \ldots, c_N \in \N_0 \\c_1 + \ldots + c_N = n}}  \prod_{k=1}^N \frac{(b_k)_{c_k}}{c_k!}   (d_k)_{c_k} \bra{\alpha(B_k) + c_k}^{(d_k - c_k)}. 
    \end{align*}
    It follows that
    \begin{align*}
        \mathcal{M}^{p, \alpha}_m \one_{B_1^{d_1} \times \cdots \times B_N^{d_N}}(\mu) &= \prod_{k=1}^N  \sum_{c_k = d_k}^\infty  \bra{1-\frac{1}{p}}^{c_k-d_k}  \frac{(b_k)_{c_k}}{c_k!}   (d_k)_{c_k} \bra{\alpha(B_k) + c_k}^{(d_k - c_k)} \\
		&= \prod_{k=1}^N   \mathscr{M}^{p, \alpha(B_k)}_{d_k}(\mu(B_k)). \qedhere
    \end{align*}
\end{proof}
\end{lemma}
We define $\hat{\mathcal{M}}^{p, \alpha}_n f_n$ as the expression on the right-hand side of \eqref{equation: explicit formula infinite-dimensional Meixner}. 

\begin{proof}[Proof of Proposition~\ref{proposition: explicit formula infinite-dimensional Meixner}]
Fix $m \in \N$ and a bounded set $B \subset \R$. We consider the function 
\begin{align}
    \label{equation: special f m}
    f_m = \one_{B_1^{d_1} \times \cdots \times B_N^{d_N}}
\end{align}
where $B_1, \ldots, B_N$ are disjoint and measurable subsets of $B$, $d_1 + \ldots + d_N = m$, and $d_1, \ldots, d_N \in \N$. We claim that \eqref{equation: explicit formula infinite-dimensional Meixner} holds true for the function $f_m$. 

Indeed, if $0 < \alpha(B_k)$ for all $k$, then applying Lemma~\ref{lemma: explicit meixner product formula} and \cite[Proposition 5.3]{intertwiningConsistent} (see also \cite[Lemma 3.1]{LYTVYNOV2003118}) yields
\begin{align*}
    \hat{\mathcal{M}}^{p, \alpha}_m f_m(\mu) &=   \prod_{k=1}^N \mathscr{M}_{d_k}^{p, \alpha(A_k)}\bra{\mu(B_k)} = \mathcal{M}^{p, \alpha}_m f_m(\mu)
\end{align*}
for $\rho_{p, \alpha}$-almost all $\mu \in \mathbf{N}$.

On the other hand, if there exists a $k \in \set{1, \ldots, N}$ such that $\alpha(B_k) = 0$, then $f_m = \one_{B_1^{d_1} \times \cdots \times B_N^{d_N}} = 0$ $\lambda_m$-almost everywhere. Therefore, using \eqref{equation: infinite-dimensional Meixner polynomials}, Proposition~\ref{proposition: factorial moment measure Pascal process} and \eqref{equation: lambda n equals kappa lambda k}, we obtain
\begin{align*}
    \int \abs{\mathcal{M}^{p, \alpha}_m f_m} \dx \rho_{p, \alpha} \leq \sum_{k=0}^m  \binom{m}{k} \bra{\frac{p}{1-p}}^m \int \widetilde{f_m} \dx \lambda_m = 0,
\end{align*}
which implies $\mathcal{M}^{p, \alpha}_m f_m = 0$ $\rho_{p, \alpha}$-almost surely. Furthermore, $\hat{\mathcal{M}}^{p, \alpha}$ is well-defined on $L^2(\lambda_n)$, as shown by the orthogonality relation~\eqref{equation: orthogonality infinite dimensional Meixner} (proved by \cite{intertwiningConsistent} or \cite{Lytvynov2003}) for $\hat{\mathcal{M}}^{p, \alpha}$. Thus, $\hat{\mathcal{M}}^{p, \alpha} f_m = 0 = \mathcal{M}^{p, \alpha}_m f_m$ $\rho_{p, \alpha}$-almost surely.

Using the \emph{functional monotone class theorem} (see e.g. \cite[Theorem 2.12.9.]{Bogachev}) we obtain that \eqref{equation: explicit formula infinite-dimensional Meixner} holds true for $f_m \in \mathcal{C}_m$ as well. More precisely, let
\begin{align*}
    \mathcal{H} &= \set{g_m : B \to \R \text{ measurable, bounded} : \mathcal{M}^{p, \alpha}_m \hat{g}_m = \hat{\mathcal{M}}^{p, \alpha}_m \hat{g}_m}
\end{align*}
where $\hat{g}_m$ is defined to be equal to $g_m$ on $B$ and equal to zero on $\R \setminus B$. Then, $\mathcal{H}$ contains the constant functions and is closed with respect the formation of uniform and monotone limits. Furthermore, the set
\begin{align*}
    \mathcal{H}_0 &= \set{\one_{C_1 \times \cdots \times C_m} : C_1, \ldots, C_m \subset B \text{ measurable}}
\end{align*}
is closed under forming products and is contained in $\mathcal{H}$, since $\widetilde{\one}_{C_1 \times \cdots \times C_m}$ can be written as a linear combination of symmetrizations of functions of the type \eqref{equation: special f m}. Therefore, by the functional monotone class theorem, $\mathcal{H}$ equals the set of all bounded, measurable functions on $B$. 
\end{proof}

\subsection{Reversible Measures for the \texorpdfstring{$n$}{n}-Particle Dynamics }

\begin{proof}[Proof of Proposition~\ref{Proposition: correlated n-particle reversible}]
To show that $\lambda^{\otimes n}$ is reversible for the $n$-particle dynamics where $\lambda$ denotes the Lebesgue measure, we use the following straightforward computation. Let $n \in \N$ and $t > 0$ be fixed. We define $Y := (Y_k)_{k=1}^n$, where $Y_k := \sqrt{1-a} B_{k,t} + \sqrt{a} B_t$ for $k \in \set{1, \ldots, n}$. By using the construction of correlated Brownian motions given in \eqref{equation: construction correlated brownian motions}, we obtain $X_{k,t} = Y_k + x_k$. Then, using Fubini's theorem, substitution, and the fact that $-Y$ is equal in distribution to $Y$, we obtain
\begin{align*}
    \int \E \sbra{f_n(Y + x)} g_n(x) \lambda^{\otimes n}(\dd x) &= \E \sbra{\int_{\R^n}  f_n(-Y + x) g_n(x) \lambda^{\otimes n}(\dd x)} \\
    = \E \sbra{\int  f_n(x) g_n(x + Y) \lambda^{\otimes n}(\dd x)} &= \int  f_n(x) \E \sbra{g_n(x + Y)} \lambda^{\otimes n}(\dd x)
\end{align*}
for all $f_n$, $g_n \in \mathcal{C}_n$.
\end{proof}

As a next step, we prove Proposition~\ref{proposition: sticky n-particle reversible}. A reversible measure for a system of $n$ ordered sticky Brownian motions is known, as established in \cite[Theorem 4.17]{brockingtonbethe}. Starting from this reversible measure, we obtain that its symmetrization is equal to the measure $\lambda_n$ up to a constant, which is the factorial moment measure of the Pascal process up to a constant.

Let $n \in \N$ be fixed. We use the notation $\Sigma_n$ for the set of \emph{partitions} of the set $\set{1, \ldots, n}$ and we define $\Pi_n$ to be the set of \emph{ordered partitions} of $\set{1, \ldots, n}$, where
\begin{align*}
    \Pi_n := \set{(a_1, \ldots, a_k) : a_1, \ldots, a_k \in \N, k \in \N, a_1 + \ldots + a_k = n}. 
\end{align*}
For each $\pi = (a_1, \ldots, a_k) \in \Pi_n$ and each function $f_n : \R^n \to \R$, we define the function $(f_n)_\pi : \R^{k} \to \R$ by
\begin{align*}
    (x_1, \ldots, x_k) \mapsto f_n(\underbrace{x_1, \ldots, x_1}_{a_1 \text{ times}}, \ldots, \underbrace{x_k, \ldots, x_k}_{a_k \text{ times}}). 
\end{align*}
We define the measure $\lambda_\pi^{\geq}$ on $\R^n$ by
\begin{align*}
    \int f_n \dx \lambda_\pi^\geq := \int \one_{\set{x_1 \geq \ldots \geq x_k}} (f_n)_\pi(x_1, \ldots x_k) \lambda^{\otimes k}(\dd (x_1, \ldots, x_k))
\end{align*}
and put
\begin{align*}
    m_\theta^{(n)} := \sum_{\pi \in \Pi_n} \theta^{\abs{\pi} - n} \bra{\prod_{A \in \pi}} \frac{1}{\abs{A}} \lambda_\pi^\geq. 
\end{align*}
We consider the mapping $\varphi : \R^n \to \set{(x_1,\ldots, x_n) \in \R^n : x_1 \geq \ldots \geq x_n}$, which orders the components of a vector in descending order. Let $(\Omega^n, \mathcal{F}^n, (Z_t^n)_{t \geq 0}$, $(\P_x^n)_{x \in \R^n})$ denote the Markov family associated with the $n$-particle dynamics as described in the proof of Lemma~\ref{lemma: correlated Brownian to unlabeled}, abbreviate $Z_t = Z_t^n$, $\E_x = \E_x^n$ and put $Y_t := \varphi(Z_t)$. 

We define the \emph{symmetrization} of $m_\theta^{(n)}$ by $\widetilde{m}_\theta^{(n)} := \frac{1}{n!} \sum_{s \in \mathfrak{S}_n} \bra{T_s}_{\#} m_\theta^{(n)}$ where $\bra{T_s}_{\#} m_\theta^{(n)}$ denotes the push-forward of $m_\theta^{(n)}$ under the mapping $T_s: (x_1, \ldots, x_n) \mapsto (x_{s(1)}, \ldots, x_{s(n)})$. $\mathfrak{S}_n$ denotes the set of permutations on $\set{1, \ldots, n}$.

\begin{proof}[Proof of~Proposition~\ref{proposition: sticky n-particle reversible}]
To prove that the push-forward measure of the symmetrization $\widetilde{m}_\theta^{(n)}$ under the mapping $(x_1, \ldots, x_n) \mapsto \delta_{x_1} + \ldots + \delta_{x_n}$ is reversible for unlabeled sticky Brownian motions $(\eta_t)_{t \geq 0}$, it suffices to verify
\begin{align*}
    \int (P_t^{[n]} f_n) g_n \dx \lambda_n = \int  ( P_t^{[n]} g_n) f_n \dx \lambda_n
\end{align*}
for symmetric $f_n, g_n \in \mathcal{C}_n$. Indeed, using $f_n \circ \varphi = f_n$ and $\varphi_{\#} \widetilde{m}_\theta^{(n)} = m_\theta^{(n)}$, we obtain
\begin{align}
    \label{equation: reduce to ordered sticky}
    \notag
    \int f_n(x) \E_x \sbra{ g_n(Z_t) } \widetilde{m}_\theta^{(n)}(\dd x) &= \int f_n(\varphi(x)) \E_x \sbra{ g_n(Z_t) }  \widetilde{m}_\theta^{(n)}(\dd x) \\
    \notag
    &= \int f_n(\varphi(x)) \E_{\varphi(x)} \sbra{ g_n(Y_t) } \widetilde{m}_\theta^{(n)}(\dd x) \\
    &= \int f_n(y) \E_{y} \sbra{g_n(Y_t) } m_\theta^{(n)}(\dd y).
\end{align}
Thereby, in the second equation, we use the strong consistency property \eqref{equation: strong consistency}. To be more precise, let $x \in \R^n$ be fixed and let $s \in \mathfrak{S}_n$ be such that $T_s(x) = \varphi(x)$. By applying \eqref{equation: strong consistency} and using the fact that $g_n \circ T_s = g_n$, we obtain
\begin{align*}
    \E_x\sbra{g_n(Z_t)} = \E_x\sbra{g_n(T_s(Z_t))} = \E_{T_s(x)} \sbra{g_n(Z_t)} = \E_{\varphi(x)} \sbra{g_n(Y_t)}. 
\end{align*}
Brockington and Warren demonstrated in \cite[Theorem 4.17]{brockingtonbethe} that $m_\theta^{(n)}$ is reversible for $n$ ordered uniform sticky Brownian motions with stickiness $\theta$. As a result, \eqref{equation: reduce to ordered sticky} is symmetric in $f_n$ and $g_n$.

Hence, to complete the proof, it only remains to show that the symmetrization of $m_\theta^{(n)}$ coincides with $\lambda_n$ up to a constant. Specifically, we claim that $\widetilde{m}_\theta^{(n)} = \frac{1}{\theta^n n!} \lambda_n$. To recall, we construct $\lambda_n$ using the measure $\alpha := \theta \lambda$ through \eqref{equation: definition lambda n}. Note that both $\widetilde{m}_\theta^{(n)}$ and $\lambda_n$ are invariant under permutation of variables. Hence, it suffices to prove that $\int f_n \dx \widetilde{m}_\theta^{(n)} = \int f_n \dx \lambda_n$
holds for all symmetric functions $f_n \in \mathcal{C}_n$.

By using the fact that the Lebesgue measure is diffuse, meaning that $\lambda(\set{x}) = 0$ for all $x \in \R$, we obtain for all $\pi=(a_1, \ldots, a_k) \in \Pi_n$
\begin{align*}
    \int (f_n)_\pi \dx \lambda^{\otimes k} &= \sum_{s \in \mathfrak{S}_k} \int \one_{\set{x_{s(1)} \geq \ldots \geq x_{s(k)}}} (f_n)_\pi(x_1, \ldots, x_k)  \lambda^{\otimes k}(\dd (x_1, \ldots, x_k)) \\
    &= \sum_{s \in \mathfrak{S}_k} \int \one_{\set{x_1 \geq \ldots \geq x_k}} (f_n)_\pi(x_{s^{-1}(1)}, \ldots, x_{s^{-1}(k)}) \lambda^{\otimes k}(\dd (x_1, \ldots, x_k)) \\
    &= \sum_{s \in \mathfrak{S}_k} \int f_n \dx \lambda_{\bra{a_{s(1)}, \ldots, s_{s(k)}}}^\geq,
\end{align*}
as $(f_n)_{\pi}$ is symmetric. 

We observe that every permutation $(a_{s(1)}, \ldots, a_{s(k)})$ of an ordered partition is itself an ordered partition. Using this fact along with \eqref{equation: definition lambda n}, we get
\begin{align*}
    \int f_n \dx \widetilde{m}_\theta^{(n)} &= \int f_n \dx m_\theta^{(n)} 
    = \sum_{\pi =(a_1, \ldots, a_k) \in \Pi_n} \theta^{k - n} \frac{1}{k!} \frac{1}{a_1 \cdots a_k} \int (f_n)_\pi \dx \lambda^{\otimes k} \\
    &= \frac{1}{ n!} \sum_{\sigma = \set{B_1, \ldots, B_k} \in \Sigma_n} \theta^{k-n} (\abs{B_1}-1)! \cdots (\abs{B_k}-1)! \int (f_n)_{\sigma} \lambda^{\otimes k} = \frac{1}{\theta^n n!} \int f_n \dx \lambda_n. 
\end{align*}
We can switch from the summation over ordered partitions $\Pi_n$ to the summation over partitions $\Sigma_n$, since the equality
\begin{align*}
    & \abs{ \set{B_1, \ldots, B_k} \in \Sigma_n : \set{\abs{B_1}, \ldots, \abs{B_k}} = \set{l_1, \ldots, l_k}  } \\
    &\hspace{4em} = \frac{n!}{k!} \frac{1}{l_1! \cdots l_k!} \abs{ (a_1, \ldots, a_k) \in \Pi_n : \set{a_1, \ldots, a_k} = \set{l_1, \ldots, l_k}  }, 
\end{align*}
holds for all $l_1, \ldots, l_k \in \N$,  $k \in \N$ with $l_1 + \ldots + l_k = n$.
\end{proof}

\subsubsection*{Acknowledgments}
I thank Sabine Jansen, Frank Redig, Simone Floreani and Jan Philipp Neumann for helpful discussions. This work was supported under Germany’s Excellence Strategy - EXC-2111 - 390814868.

\newcommand{\etalchar}[1]{$^{#1}$}
\providecommand{\bysame}{\leavevmode\hbox to3em{\hrulefill}\thinspace}
\providecommand{\MR}{\relax\ifhmode\unskip\space\fi MR }
\providecommand{\MRhref}[2]{%
  \href{http://www.ams.org/mathscinet-getitem?mr=#1}{#2}
}
\providecommand{\href}[2]{#2}

\end{document}